\documentclass[11pt,a4paper]{article}
\usepackage[T1]{fontenc}
\usepackage[utf8]{inputenc}
\usepackage{lmodern,amsmath,amssymb,amsthm,mathrsfs,microtype}
\usepackage[margin=28mm]{geometry}
\usepackage[unicode,hidelinks]{hyperref}
\usepackage{bookmark}
\usepackage{tikz}
\usetikzlibrary{arrows.meta,decorations.markings}
\hypersetup{pdftitle={Uniqueness of tangent cones with immersed special Legendrian links}}
\numberwithin{equation}{section}
\newtheorem{theorem}{Theorem}[section]
\newtheorem{proposition}[theorem]{Proposition}
\newtheorem{lemma}[theorem]{Lemma}

\theoremstyle{definition}
\newtheorem{definition}[theorem]{Definition}
\theoremstyle{remark}

\newtheorem*{remark*}{Remark}
\newtheorem{example}[theorem]{Example}
\newtheorem*{example*}{Example}
\newtheorem*{notation}{Notation}
\newcommand{\R}{\mathbb R}
\newcommand{\C}{\mathbb C}
\newcommand{\dist}{\operatorname{dist}}
\newcommand{\Ree}{\operatorname{Re}}
\newcommand{\dd}{\,d}
\allowdisplaybreaks[2]
\title{Uniqueness of tangent cones\\with immersed special Legendrian links}
\author{Yang Li}
\date{\today}
\begin{document}
	\maketitle
	
	\begin{abstract}
		Under the assumption of integrability of the smooth Jacobi fields
		and a condition on the intersection pattern of the link, we prove uniqueness of
		special Lagrangian tangent cones in $\C^3$ whose links are compact
		immersed special Legendrian surfaces with transverse double points.
		The proof is based on a discrete \L{}ojasiewicz type inequality, which
		comes from a gluing obstruction mechanism.
	\end{abstract}
	
	\tableofcontents
	
	\section{Introduction}\label{sect:introduction}
	
	A fundamental problem in minimal surface theory is to prove uniqueness
	of the tangent cone at a singular point of an area-minimizing current.
	Three basic approaches are as follows.
	\begin{enumerate}
		\item \emph{Allard--Almgren \cite{AllardAlmgren}.}
		Suppose that one tangent cone has multiplicity one and a smooth compact
		link, and assume as an additional hypothesis that every Jacobi field of
		the link is integrable through minimal submanifolds of the sphere.
		On fixed annuli, regularity gives strong convergence of the normalized
		normal graphs to a Jacobi field. Integrability provides a nearby minimal
		cone corresponding to its degree-one part. Relative to this comparison
		cone, estimates for the remaining components on successive annuli give
		uniqueness of the tangent cone and convergence at a polynomial rate.
		
		\item \emph{Simon's \L{}ojasiewicz method \cite{Simon83}.}
		For a multiplicity-one tangent cone with a smooth compact link, analyticity
		of the area functional and an infinite-dimensional gradient inequality
		control the accumulated change under rescaling. This gives uniqueness
		without the Allard--Almgren integrability hypothesis, though the general
		rate need not be polynomial.
		
		\item \emph{Simon's cylindrical-cone method \cite{Simon93,Simon94}.}
		For a multiplicity-one cylindrical tangent cone $C\times\R^k$ in
		codimension one, the link is singular along the spine. Simon proves
		uniqueness and polynomial convergence when $C$ is smooth away from the
		vertex, strictly stable and strictly minimizing, and satisfies his
		conditions on homogeneous Jacobi fields: the fields of degrees zero and one
		are generated by translations and rotations, and certain negative
		integral homogeneities are absent. Nonconcentration near the spine then
		makes Jacobi-field approximation and a three-annulus argument effective.
	\end{enumerate}
	
	Sz\'ekelyhidi proves uniqueness for the nonintegrable cylinder
	$C(S^3\times S^3)\times\R$ \cite{Sze20}. In place of nearby minimal
	cones, he constructs minimal comparison hypersurfaces on large annuli
	as logarithmic perturbations of cones. His nonconcentration estimate,
	measured in an adapted $L^\infty$ distance, bounds the error relative to
	a comparison surface on an interior annulus by the graphical error away
	from the singular axis, plus an arbitrarily small multiple of the error
	on the larger annulus. Related estimates enter his strong unique
	continuation theorem at cylindrical tangent cones \cite{Sze21}.
	Collins--Li adapt the strategy to exact special
	Lagrangians: if a multiplicity-one tangent cone is
	$\R^k\times\mathbf C$, where $\dim\mathbf C>2$, the link of
	$\mathbf C$ is smooth and connected, and $\mathbf C$ satisfies their
	integrability hypothesis, then the tangent cone is unique \cite{CL}.
	Their argument uses local potentials together with nonconcentration and
	decay estimates.
	
	We consider a closed integral three-current $L$ of multiplicity one in
	$\C^3$, which is \emph{special Lagrangian} with respect to the standard Euclidean structure
	$(J,\omega,\Omega)$, namely
	\[
	\omega|_L=0,\qquad \operatorname{Im}\Omega|_L=0.
	\]
	By an observation of Harvey and Lawson
	\cite[Theorem~II.4.2]{HL}, the fact that $L$ is calibrated by
	$\Ree\Omega$ implies that $L$ is locally volume minimizing.
	We study special Lagrangians whose
	tangent cones have immersed surface links with transverse double points.
	The simplest example is the union of two 3-planes meeting along a line.
	
	Our main contribution is a new mechanism which proves the uniqueness of the
	tangent cone even when $L^2$ concentration occurs. The normalized height
	may concentrate near the rays through the double points of the link,
	and hence disappear from the Jacobi field limit on regular regions.
	Instead, we turn a gluing obstruction into a discrete \L{}ojasiewicz inequality.
	This controls the drift of the best approximating cone over scales.
	
	We identify an immersed Legendrian with its image unless its domain is
	needed. Let $\Lambda\subset S^5$ be a compact oriented
	smooth immersed special Legendrian surface of phase zero, and let
	$F:\Sigma\to\Lambda$ denote its parametrization. The domain $\Sigma$
	may be disconnected even when its image $\Lambda$ is connected. We allow
	finitely many double points,
	transverse in the contact distribution, and no higher intersections.
	Write $\mathbf C_0=\mathbf C(\Lambda)=\{Rp:R\ge0,\ p\in\Lambda\}$
	for its special Lagrangian cone.
	
	Write $\Sigma=\coprod_{v\in V}\Sigma_v$. The \emph{intersection graph}
	$G(\Lambda)=(V,E)$ has one vertex for each connected component $\Sigma_v$
	and one unoriented edge for each double point, joining the vertices
	corresponding to the components containing its two preimages; loops and
	parallel edges are allowed.
	
	\begin{definition}\label{def:graph-condition}
		We say that $\Lambda$ satisfies the \emph{intersection condition} if every edge
		of $G(\Lambda)$ belongs to at most one simple cycle.
		(Equivalently, each connected component of $G(\Lambda)$ is a \emph{cactus graph}.
		A \emph{simple cycle} is a closed path with no repeated edges and no repeated
		vertices except its starting and ending vertex.
		Loops and pairs of parallel edges are regarded as simple cycles of length one
		and two, respectively.)
	\end{definition}
	
	We assume that $\Lambda$ is \emph{Jacobi-integrable}:
	there are a small finite-dimensional parameter neighbourhood $\mathcal M$
	of $0$ and a family $\{\Lambda_s\}_{s\in\mathcal M}$ of special
	Legendrian immersions, with $\Lambda_0=\Lambda$, whose infinitesimal normal
	variations span the smooth Jacobi fields on $\Sigma$. Choose smooth
	parametrizations $F_s:\Sigma\to\Lambda_s$, and write
	$\mathbf C_s=\mathbf C(\Lambda_s)$. The family need not preserve the
	double points; no matching condition is imposed on their two preimages.
	
	Our main result is
	\begin{theorem}\label{thm:target}
		Under these assumptions, suppose $L$ has $\mathbf C_0$ as a tangent cone
		at the origin and $\Lambda$ satisfies the intersection condition in
		Definition~\ref{def:graph-condition}. Then $\mathbf C_0$ is the unique tangent
		cone of $L$ at the origin.
	\end{theorem}
	
	\begin{example}\label{ex:tree}
		The hypotheses hold for a union of great special Legendrian $2$-spheres
		meeting only in transverse double points, provided $G(\Lambda)$ is
		obtained from a forest by replacing each edge with two parallel edges.
		These two edges correspond to the two antipodal double points of an
		intersecting pair of $2$-spheres. They form the only simple cycle
		containing either edge, so the intersection condition holds.
		The smooth Jacobi fields of each sphere integrate to independent plane
		deformations.
	\end{example}
	
	\begin{samepage}
		A byproduct of the proof is local rigidity of the special Legendrian link.
		\begin{theorem}\label{thm:rigidity}
			Suppose $\Lambda$ is Jacobi-integrable and satisfies the intersection
			condition in Definition~\ref{def:graph-condition}. Then every
			special Legendrian integral two-cycle $\Lambda'$ in $S^5$
			sufficiently close to $\Lambda$ in the varifold topology equals
			$\Lambda_s$ for some $s\in\mathcal M$ near $0$.
		\end{theorem}
	\end{samepage}
	
	\subsection{The main estimates}
	
	The quantity $D(R)$ measures the normalized $L^2$ height from the best
	approximating cone:
	\begin{equation}\label{eqn:optimized-height}
		D(R)^2=\inf_{s\in\mathcal M} R^{-5}\int_{L\cap(B_{4R}\setminus\overline B_{R/4})}
		\dist(x,\mathbf C_s)^2\dd\|L\|(x).
	\end{equation}
	Put
	$\Theta_0=\|\mathbf C_0\|(B_1)$ and
	\begin{equation}\label{eqn:radial-excess}
		m(R)=R^{-3}\|L\|(B_R)-\Theta_0,
		\qquad E(R)=m(4R)-m(R/4).
	\end{equation}
	When $\mathbf C_0$ is a tangent cone, $m\ge0$ and $m(R)\to0$ as $R\downarrow0$.
	By the monotonicity formula (\ref{eqn:monotonicity}), $E(R)$ is a weighted
	integral of the squared normal component of the radial vector. It
	measures the failure of radial invariance.
	
	The principal estimates are
	\begin{equation}\label{eqn:two-main-estimates}
		m(R)\le C\{D(R)^2+E(R)\},\qquad
		E(R)\ge C^{-1}D(R)^{2q}\quad\text{for some }1<q<2.
	\end{equation}
	The second estimate applies also to nearby conical currents and directly
	implies Theorem~\ref{thm:rigidity}. Together the estimates give
	$E(R)\ge C^{-1}m(R)^q$. Since $q<2$, a numerical iteration makes the sum of
	$\sqrt{E(R)}$ over dyadic scales finite. This controls
	the accumulated drift of the best approximating cone and gives uniqueness
	of the tangent cone; see Section~\ref{sect:uniqueness-tangent-cone}.
	
	\subsection{Proof sketch for the discrete \L{}ojasiewicz estimate}
	\label{sect:two-plane-proof-sketch}
	
	We explain the lower bound for $E$ in (\ref{eqn:two-main-estimates}),
	first when $\mathbf C_0$ is the union of two planes meeting along a line.
	The comparison link $\Lambda_s$ is then the union of two great special
	Legendrian $2$-spheres. Work at unit scale and suppose that
	$0<E\ll D^{2q}\ll1$.
	A normalized Jacobi field limit, together with integrability, rules
	out the case in which the height does not concentrate near the
	intersection line.
	
	In the main case of concentration, away from the intersection line
	$L$ is graphical over a nearby plane pair, while near the line it has local pieces
	modelled on a Lawlor neck times a line. A suitable radial slice
	consists of graphs over the comparison spheres joined by small Lawlor
	necks near the two antipodal double points of the original link.
	Small $E$ requires this slice
	to be nearly special Legendrian. These pieces of $L$ therefore give an
	approximate solution to a special Legendrian gluing problem.
	
	Let $\rho$ be a radius comparable to the largest neck radius.
	The largest logarithmic coefficient of the Lagrangian
	potentials is then comparable to $\rho^2$.
	One can show that $\rho^2\ge C(\varepsilon)^{-1}D^{1+\varepsilon}$
	for arbitrarily small $\varepsilon>0$ (cf.\ Proposition~\ref{prop:stopping}).
	
	\paragraph{The gluing obstruction.}
	On a special Legendrian surface every period of
	$\lambda=\iota_x\omega$ is zero. On $L$, $|\lambda|=|x^\perp|$, so
	$E$ controls the period around a loop $\gamma$ crossing one neck and returning
	through the other, following the graphical sheets in between.
	Flux balance gives opposite logarithmic coefficients at the two
	punctures of each sphere, to leading order. Since the loop crosses
	the necks in opposite directions, their contributions add, with total
	magnitude comparable to $\rho^2\log(1/\rho)$.
	
	The two comparison spheres need not intersect after deformation.
	Near each original double point, they are joined by a short Reeb chord.
	Let $g(s)$ be the signed sum of the integrals of $\lambda$ along
	these chords. The period formula expresses
	$\int_\gamma\lambda$ as the sum of $g(s)$, the logarithmic neck
	contribution and an error. Thus
	\[
	|g(s)|\ge C^{-1}\rho^2\log(1/\rho)
	-\left|\int_\gamma\lambda\right|
	-o\bigl(\rho^2\log(1/\rho)\bigr).
	\]
	For $E\le\rho^5$, we have $|\int_\gamma\lambda|=O(\rho^2)$,
	so for small $\rho$,
	\[
	|g(s)|\ge C^{-1}\rho^2\log(1/\rho).
	\]
	The two chords are antipodal images. Since the antipodal map preserves
	$\lambda$, their actions agree, but the loop traverses them in opposite
	directions. Thus $g(s)=0$, a contradiction.
	
	We now discuss an alternative approach that can be adapted to the general
	case. The function $g$ is analytic with $g(0)=0$. We rule out cancellation
	by comparing $g(s)$ with $dg/ds$.
	Under a deformation represented
	by a smooth Jacobi function $h$, the derivative of a chord action is
	twice the difference between the values of $h$ at its endpoints. Testing the flux balance
	against $h$ therefore controls $dg/ds$.
	For a single simple cycle, this gives
	\[
	\left|\frac{dg}{ds}\right|
	\le C\{\rho^2\log(1/\rho)+\rho^{-2}\sqrt E\}.
	\]
	The analytic gradient inequality gives
	$|g|^\theta\le C|dg/ds|$ for some $0<\theta<1$. The lower bound for
	$|g(s)|$ and the upper bound for $|dg/ds|$ contradict this when $E$
	is too small. When there are several simple cycles,
	the intersection condition prevents cancellation of the logarithmic
	terms in a suitable weighted sum of their actions. Applying the
	gradient inequality to this sum gives the quantitative obstruction
	$E\ge\rho^{q_0}$ for some $q_0<8$; see Proposition~\ref{prop:assembly}.
	The lower bound for $\rho$ then gives
	$E\ge C^{-1}D^{2q}$ with $q<2$, as required.
	
	\paragraph{Acknowledgements.}
	The author is supported by the Royal Society URF.
	During this work, AI suggested the graph theoretic condition for producing
	a gluing obstruction, which substantially simplified an earlier draft involving a linear system
	of equations coming from the loop period integrals and flux balancing
	conditions.

	\section{Background}\label{sect:background}
	
	\subsection{The monotonicity formula and the boundary formula}
	
	\begin{notation}
		For $z\in\C^3$ and $\rho>0$, let
		$\lambda_{z,\rho}(x)=(x-z)/\rho$ and
		$T_{z,\rho}=(\lambda_{z,\rho})_\#T$.
	\end{notation}
	
	Write $\|T\|$ for the mass measure of $T$ and $x^\perp$ for the
	component of the position vector $x$ orthogonal to the approximate
	tangent plane of $T$ at $x$.
	
	\begin{proposition}\label{prop:monotonicity}
		For a calibrated integral three-current $T$ without boundary in $B_{R_0}$
		and $0<R_1<R_2<R_0$,
		\begin{equation}\label{eqn:monotonicity}
			R_2^{-3}\|T\|(B_{R_2})-R_1^{-3}\|T\|(B_{R_1})
			=\int_{R_1<|x|<R_2}|x^\perp|^2|x|^{-5}\dd\|T\|.
		\end{equation}
		For almost every $R<R_0$,
		\begin{equation}\label{eqn:stokes-boundary-integral}
			3\|T\|(B_R)=\int_{\langle T,|x|,R\rangle}
			\iota_x\Ree\Omega.
		\end{equation}
		Here $\langle T,|x|,R\rangle$ is the current slice, with its boundary
		orientation.
	\end{proposition}
	
	\begin{proof}
		The first identity is the monotonicity formula \cite{Allard}.
		The second follows from calibration, the Stokes formula and
		$d(\iota_x\Ree\Omega)=3\Ree\Omega$.
	\end{proof}
	
	\subsection{Contact geometry and Jacobi fields}
	\label{sect:contact-geometry-jacobi-fields}
	
	Write $\lambda=\iota_x\omega$ for the Liouville one-form.
	The restriction of $\lambda$ to $S^5$ is a contact form with Reeb
	vector field $x\mapsto Jx$.
	An immersion $F:\Sigma\to S^5$ is \emph{Legendrian} if $F^*\lambda=0$,
	and \emph{special Legendrian} if its cone is special Lagrangian.
	A special Legendrian immersion is minimal in the sphere
	\cite[Sections~2 and~3.4]{Joyce}.
	
	We use the Laplacian $\Delta_\Sigma=\operatorname{div}_\Sigma\nabla_\Sigma$
	and write $R=|x|$.
	On the cone, the Jacobi equation for a normal field $V=J\alpha^\sharp$ is
	$d\alpha=d^*\alpha=0$ \cite[Section~5.1]{Joyce}.
	A degree-one Jacobi field has the form $V=J\nabla(R^2\phi)$.
	
	\begin{lemma}\label{lem:jacobi}
		Finite-energy degree-two Jacobi potentials are precisely $R^2\phi$ with
		$\phi\in\ker(\Delta_\Sigma+6)$. They extend smoothly across the
		two distinct preimages in $\Sigma$ of each double point of $\Lambda$,
		separately on each local sheet.
	\end{lemma}
	
	\begin{proof}
		Finite energy gives $\phi\in H^1$ near each double-point preimage.
		These points have zero capacity in the smooth surface $\Sigma$, so
		the weak equation extends across each point. Elliptic regularity gives the
		assertions.
	\end{proof}
	
	By the Jacobi-integrability assumption, the space of smooth Jacobi fields
	has constant dimension for $s\in\mathcal M$.
	
	\begin{example*}[Jacobi fields for unions of planes]
		For a round $S^2$, the eigenspace $\ker(\Delta_{S^2}+6)$ consists
		of the restrictions of trace-free quadratic polynomials on $\R^3$.
		The corresponding Jacobi fields integrate to deformations through
		phase-zero special Lagrangian 3-planes. The planes deform independently,
		so unions of great special Legendrian spheres are Jacobi-integrable.
	\end{example*}
	
	\subsection{A linear three-annulus estimate}
	
	In $\R_t\times\R_y^2$, for $K>1$ put
	\[
	U_r=\{(t,y):|t|<r,\ r<|y|<Kr\},\qquad
	\|v\|_r^2=r^{-5}\int_{U_r}|v|^2.
	\]
	
	\begin{lemma}\label{lem:linear-three-annulus}
		Fix $3<\gamma<4$. For sufficiently large $K$ and any $r>0$,
		a nonzero $t$-independent closed and coclosed one-form $v$ on
		$\R\times\{y\in\R^2:r/K<|y|<K^2r\}$
		satisfies
		\[
		\|v\|_r^2\ge K^3\|v\|_{K r}^2
		\quad\Longrightarrow\quad
		\|v\|_{r/K}^2>2K^{\gamma}\|v\|_r^2.
		\]
	\end{lemma}
	
	\begin{proof}
		The proof is modelled on the three-annulus arguments in \cite{ES,Sze21}.
		Write $v=c(y)\,dt+a(y)\,dy_1+b(y)\,dy_2$.
		Closedness and coclosedness say that $c$ is constant and $a-ib$ is
		holomorphic in $y_1+iy_2$. Integrating in $t$ gives
		\[
		\|v\|_\rho^2
		=2\rho^{-4}\int_{\rho<|y|<K\rho}(|a|^2+|b|^2+|c|^2)\,dy.
		\]
		The Laurent expansion of $a-ib$ decomposes the transverse part into
		orthogonal homogeneous terms of integer degrees $k$. A term of degree
		$k$ has squared norm proportional to $\rho^{2k-2}$.
		
		Let $b_k^2$ be the contribution of the $k$-th angular Fourier mode to
		the squared norm at scale $r$, with
		the contribution of $c\,dt$ included in $b_0^2$.
		Parseval's identity and dilation give
		\[
		\|v\|_r^2=\sum_k b_k^2,\qquad
		\|v\|_{K r}^2=\sum_k K^{2k-2}b_k^2,\qquad
		\|v\|_{r/K}^2=\sum_k K^{2-2k}b_k^2.
		\]
		The hypothesis gives
		\[
		\sum_{k\ge0}b_k^2\le K^2\|v\|_{K r}^2
		\le K^{-1}\|v\|_r^2.
		\]
		Hence
		\[
		\|v\|_{r/K}^2\ge K^4\sum_{k\le-1}b_k^2
		\ge K^4(1-K^{-1})\|v\|_r^2
		>2K^{\gamma}\|v\|_r^2
		\]
		for sufficiently large $K$, since $\gamma<4$.
	\end{proof}
	
	\subsection{Classification of transverse limits with planar ends}
	
	We recall the classification of transverse special Lagrangians with
	planar ends and the asymptotic coefficients of Lawlor necks
	\cite{Lawlor,LSSz}.
	
	\begin{lemma}\label{lem:lawlor-classification}
		Let $T$ be a boundaryless calibrated two-current in $\C^2$ with quadratic
		mass growth. Suppose that outside a compact set, $T$ consists of
		two multiplicity-one normal graphs $V_i$ over affine special
		Lagrangian planes $Q_i$, $i=1,2$. Assume that the planes have distinct
		directions and $|V_i(x)|\to0$ as $|x|\to\infty$. Then
		$T=Q_1\cup Q_2$, or $T$ is a Lawlor neck asymptotic to
		$Q_1\cup Q_2$.
	\end{lemma}
	
	\begin{proof}
		Write $z_j=x_j+iy_j$. Upon taking the hyperk\"ahler rotation
		$w_1=x_1+ix_2$, $w_2=y_1-iy_2$, the K\"ahler form is
		$\Ree(dz_1\wedge dz_2)$, and the holomorphic volume form is
		$\omega-i\operatorname{Im}(dz_1\wedge dz_2)$.
		The current is therefore a positive complex curve.
		In the projective compactification each decaying graphical end meets
		the line at infinity at one point with intersection multiplicity one.
		The graph description gives finite projective area there, so the
		removable-singularity theorem for analytic curves closes the end.
		The resulting projective curve has degree two and, by Chow's theorem,
		is the union of two lines or an irreducible conic. In suitable complex
		coordinates $Z_1,Z_2$, the conic has equation $Z_1Z_2=a+ib$, with $a,b\in\R$,
		and is a Lawlor neck.
	\end{proof}
	
	\subsection{A criterion for Lawlor neck extraction}
	
	Fix a compact family of pairs of affine special Lagrangian $3$-planes
	$\Pi_\pm=\mathbb R\vec v+z_\pm+Q_\pm$, where $\vec v$ is a unit vector,
	$z_\pm\perp\vec v$, and $Q_-,Q_+\subset(\mathbb C\vec v)^\perp$ are
	uniformly transverse. Fix $K>1$ and
	$C_1>1+\max_\pm|z_\pm|$ uniformly over this family.
	Write $x=t\vec v+z_\pm+y$, where $y\in Q_\pm$, and set $r=|y|$.
	
	\begin{lemma}\label{lem:effective-lawlor}
		For any fixed ball $B$ centred at $0$, there is $C>1$ such that, for every
		$\varepsilon>0$, there are $C_2(\varepsilon)>C_1+2K$ and $\varepsilon_1>0$ with the following property.
		Let $T$ be a closed special Lagrangian integral current in
		$B_{C_2(\varepsilon)}\supset\overline B$. On
		$\{r>C_1,\ |t|<r/(8K)\}\cap B_{C_2(\varepsilon)}$, assume $T$ can be written
		as a multiplicity-one normal graph $V$ over the affine planes, with
		$|V|+r|\nabla V|+r^2|\nabla^2V|\le C_1/r$.
		If
		\[
		\begin{gathered}
			\|T\|(B_{C_2(\varepsilon)})\le C_1C_2(\varepsilon)^3,\qquad
			\int_{B_{C_2(\varepsilon)}}|\vec v^\perp|^2\,d\|T\|\le\varepsilon_1,\\
			\int_{B_{C_1+2K}}\dist(x,\Pi_-\cup\Pi_+)^2\,d\|T\|(x)\ge C_1^{-1},
		\end{gathered}
		\]
		then $T$ in $B$ is a multiplicity-one $C^2$ graph of norm at most
		$\varepsilon$ over a Lawlor neck times $\R$, with $C^{-1}\le |a+ib|\le C$.
	\end{lemma}
	
	\begin{proof}
		We suppose for contradiction that the assertion fails for some fixed
		$B,\varepsilon$. Choose counterexamples $T_i$ on $B_i$, with
		$\int_{B_i}|\vec v_i^\perp|^2\,d\|T_i\|\to0$.
		By the compactness of minimal surfaces, after passing to a subsequence
		for which the affine planes converge, $T_i$ converges locally to $T_\infty$
		as currents and varifolds.
		The bound on $\vec v^\perp$ makes $T_\infty$ translation invariant
		along $\vec v$, so $T_\infty=\R\vec v\times S$.
		
		The lower distance bound passes to the limit.
		The exterior graphical bound gives smooth convergence on compact subsets
		there and $|V_\infty|\le Cr^{-1}$.
		Thus $S$ has two multiplicity-one ends decaying to the prescribed
		affine planes, and the mass bound gives quadratic mass growth.
		
		The Lagrangian condition makes the $J\vec v$ coordinate constant on
		each component of $S$, which is calibrated by
		$\iota_{\vec v}\Ree\Omega$ in a copy of $\C^2$.
		A component with one end is its affine asymptote by the monotonicity formula.
		Lemma~\ref{lem:lawlor-classification} therefore shows that $S$ is either
		the union of the two affine planes or a Lawlor neck. Any additional compactly supported
		component is zero by the monotonicity formula. The planar alternative contradicts
		the positive distance integral.
		
		Evaluating $Z_1Z_2=a+ib$ on a fixed exterior graphical region bounds
		$|a+ib|$ above. The lower distance bound keeps it uniformly away from
		zero, since necks with $a+ib\to0$ converge to the two affine planes.
		
		Allard regularity and interior estimates now give $C^2$ graphical
		convergence to the product in $B$, a contradiction.
	\end{proof}
	
	\begin{remark*}
		With only one affine plane and without the lower distance bound, the
		conclusion is a normal graph over that plane with $C^2$ norm at most
		$\varepsilon$.
		The same conclusions hold if the normal graphs $V$ are taken over
		submanifolds that are themselves $C^4$ graphs of norm at most $\varepsilon_1$
		over the affine planes in $B_{C_2(\varepsilon)}$.
		These submanifolds have the same affine limits, so the proof is unchanged.
	\end{remark*}

	\section{Graphical extension and Lawlor neck extraction}\label{sect:graphical-extension-lawlor-neck-extraction}
	
	Throughout this section and Section~\ref{sect:gluing-obstruction} we work at
	unit scale and write $D=D(1)$ and $E=E(1)$.
	Our initial smallness
	conditions are
	\begin{equation}\label{eqn:initial-closeness}
		\begin{gathered}
			\int_{L\cap(B_4\setminus\overline B_{1/4})}\dist(x,\mathbf C_0)^2\dd\|L\|(x)
			+\int_{\mathbf C_0\cap(B_4\setminus\overline B_{1/4})}
			\dist(x,\operatorname{spt}L)^2\dd\|\mathbf C_0\|(x)<\delta^2,\\
			\|L\|(B_4)<\|\mathbf C_0\|(B_4)+\delta.
		\end{gathered}
	\end{equation}
	We use $\delta$ for initial closeness and $\rho_0$ for the smallness
	threshold on stopping radii.
	For sufficiently small $\delta$, fix a \emph{best approximating cone}
	$\mathbf C_s$, with $s\in\operatorname{int}\mathcal M$, attaining
	the infimum defining $D$.
	
	On small balls in the regular part of $\mathbf C_0$, sufficiently small
	$L^2$ height and mass ratios close to one give multiplicity-one normal
	graphs by Allard regularity. Interior estimates then control the $C^1$
	norm by the $L^2$ height on slightly larger balls. Conditions
	(\ref{eqn:initial-closeness}) ensure this on fixed compact regions of
	$\mathbf C_0\cap(B_4\setminus\overline B_{1/4})$ away from the singular
	rays. The task of this section is to extend these graphs towards the
	singular rays and describe $L$ in their neighbourhoods.
	
	\subsection{Reduction to the concentrated case}
	\label{sect:concentration-reduction}
	
	We show that small $E/D^2$ forces concentration of the normalized
	height near the rays through the double points of the link.
	
	Let $\pi:S^5\to\mathbb{CP}^2$ be the Hopf projection. Near each original double point, transversality gives a unique
	nearby intersection of the two projected branches of $\Lambda_s$, depending
	smoothly on $s$. Let $q_-,q_+$ be the corresponding points on the two
	local sheets of $\Lambda_s$.
	Choose the parametrizations $F_s$ so that their preimages in $\Sigma$
	are fixed, and denote this finite set by $\mathcal P$. Let $\mathcal R_s$
	be the union of the rays through these points.
	
	\begin{lemma}\label{lem:nonconcentration}
		For every $\varepsilon>0$ and $0<\tau<1/8$, there are $\delta,\varepsilon_1>0$,
		independent of $s$, such that, if (\ref{eqn:initial-closeness}) holds,
		$\mathbf C_s$ is a best
		approximating cone with $s\in\operatorname{int}\mathcal M$,
		$D>0$ and $E\le\varepsilon_1 D^2$, then
		$L$ is given by normal graphs $V$ over the part of $\mathbf C_s$ with
		$1/2<|x|<3$ and $\dist(x,\mathcal R_s)>\tau/2$, and
		\[
		\int_{\mathbf C_s\cap\{1/2<|x|<3,\ \dist(x,\mathcal R_s)>\tau\}}|V|^2\le\varepsilon D^2.
		\]
	\end{lemma}
	
	\begin{proof}
		Suppose the assertion fails for a sequence $L_i$ with $\delta_i\downarrow0$,
		$E_i/D_i^2\to0$,
		and best approximating cones $\mathbf C_{s_i}$.
		Compactness and (\ref{eqn:initial-closeness}) give a stationary integral
		varifold limit supported on $\mathbf C_0\cap(B_4\setminus\overline B_{1/4})$.
		The second height integral in (\ref{eqn:initial-closeness}) and the
		monotonicity formula ensure that every sheet is present. Constancy, monotonicity and the mass bound give multiplicity one, so
		$L_i\to\mathbf C_0$ as varifolds on $B_4\setminus\overline B_{1/4}$.
		Allard regularity gives the normal graphs $V_i$ on this region
		and on every other compact regular subregion of
		$B_4\setminus\overline B_{1/4}$.
		For a sufficiently small $C^1$ normal graph
		$Y(R,p)=RF_s(p)+V(R,p)$, the vector
		$V-R\partial_RV=Y-R\partial_RY$ is normal to the cone, and its
		projection onto the normal space of the graph is $Y^\perp$. Hence
		\begin{equation}\label{eqn:radial-graph-equation}
			|V-R\partial_RV|
			=R^2\left|\partial_R\left(\frac VR\right)\right|
			\le C|Y^\perp|.
		\end{equation}
		By the interior estimates, a subsequence of $V_i/D_i$ converges
		smoothly on compact regular regions to a Jacobi field $V_\infty$.
		The inequality in (\ref{eqn:radial-graph-equation}) and $E_i/D_i^2\to0$ give
		$R\partial_RV_\infty=V_\infty$. Thus $V_\infty=J\nabla(R^2\phi)$.
		Fatou's lemma and the contradiction hypothesis imply $V_\infty\in L^2$.
		Lemma~\ref{lem:jacobi} now removes the punctures and gives
		$\phi\in\ker(\Delta_\Sigma+6)$ on the compact domain $\Sigma$.
		
		Since each $\mathbf C_{s_i}$ is a best approximating cone, integrability
		implies that $V_\infty$ is orthogonal to every smooth degree-one Jacobi
		field in the annular $L^2$ inner product.
		(The contribution of the $L^2$ pairing from the concentrated region is negligible because,
		with
		\[
		W_{\tau'}=\{x\in B_4\setminus\overline B_{1/4}:
		\dist(x,\mathcal R_0)<\tau'\},
		\]
		the Cauchy--Schwarz inequality bounds it by
		\[
		\frac{C}{D_i}\int_{W_{\tau'}}\dist(x,\mathbf C_{s_i})\dd\|L_i\|(x)
		\le C\|L_i\|(W_{\tau'})^{1/2}.
		\]
		By varifold convergence, the right-hand side tends to zero as
		$i\to\infty$ and then $\tau'\downarrow0$.) Since $V_\infty$ is a smooth
		degree-one Jacobi field, it is orthogonal to itself. Hence
		$V_\infty=0$, contradicting the lower height bound in the contradiction hypothesis.
	\end{proof}
	
	\subsection{Qualitative graphical extension}
	
	Our goal is to extend the normal graphs closer to $\mathcal R_s$.
	We prove a local criterion under approximate translation invariance.
	
	For a unit vector $\vec v\in\C^3$, let
	$Q_-,Q_+\subset(\C\vec v)^\perp$ be transverse real two-planes
	calibrated by $\iota_{\vec v}\Ree\Omega$.
	Fix a compact family $\mathcal F$ of such triples $(\vec v,Q_-,Q_+)$.
	For $z_-,z_+\in\R J\vec v$, set
	\[
	P_\pm=\R\vec v+z_\pm+Q_\pm,\qquad
	\ell_\pm=\R\vec v+z_\pm.
	\]
	Thus $P_-,P_+$ are affine special Lagrangian 3-planes. They intersect
	along $\ell_-=\ell_+$ if $z_-=z_+$, and are disjoint otherwise.
	On each plane write $x=t\vec v+z_\pm+y$, with $y\in Q_\pm$, and
	recall $U_1=\{(t,y):|t|<1,\ 1<|y|<K\}$
	(cf.\ Lemma~\ref{lem:linear-three-annulus}).
	
	\begin{lemma}\label{lem:qualitative-continuation}
		Fix a sufficiently large $K$, and let $C_1>1$ and $\eta>0$.
		There is $\varepsilon_0>0$ with
		the following property. Let $(\vec v,Q_-,Q_+)\in\mathcal F$, assume
		$|z_-|,|z_+|\le C_1$, and define $P_\pm,\ell_\pm$ as above.
		Let $T$ be a closed special Lagrangian integral current in $B_K$ with
		$\|T\|(B_K)\le C_1$. Suppose that outside the unit tubes
		about $\ell_-,\ell_+$, $T$ is given by multiplicity-one normal graphs
		$V_\pm$ over $P_\pm$. Assume these fields extend as special Lagrangian
		graphs with $C^2$ norm at most one to
		$\{|t|<2K,\ 1/2<|y|<2K\}$ on each plane, and satisfy
		\[
		\max_\pm\|V_\pm\|_{L^2(U_1)}\le\varepsilon_0,\qquad
		\int_{B_K}|\vec v^\perp|^2\dd\|T\|\le\varepsilon_0^2.
		\]
		Then the graphical representation extends to $B_4$ outside the tubes of radius $1/(2K)$
		about $\ell_-,\ell_+$, with
		$\max_\pm\|V_\pm\|_{C^2}\le\eta$.
	\end{lemma}
	
	\begin{proof}
		We follow the compactness argument of \cite[Proposition~3.2]{CL}.
		Suppose the conclusion fails as $\varepsilon_0\downarrow0$.
		The parameter and mass bounds give a subsequence with convergent model
		parameters and $T_i\to T_\infty$ as currents and varifolds.
		
		The uniform $C^2$ bound and interior elliptic estimates give $C^2$
		convergence of the graphical fields on compact subsets of
		$\{|t|<2K,\ 1/2<|y|<2K\}$.
		Their limits vanish on $U_1$ by the $L^2$ smallness, and hence on this
		connected domain by unique continuation. The portions of $P_\pm\cap B_K$
		outside the unit tubes are compactly contained in this domain, so the
		fields tend to zero there in $C^2$.
		Thus $T_\infty$ is closed and calibrated, and agrees with $[P_-]+[P_+]$
		in $B_K$ outside the unit tubes about the two axes.
		
		The vanishing of $\int_{B_K}|\vec v_i^\perp|^2\dd\|T_i\|$ implies
		that $\vec v$ is tangent to $T_\infty$ almost everywhere in $B_K$.
		Hence $T_\infty$ is translation invariant along the $\vec v$ direction.
		The singular set has Hausdorff dimension at most one and does not
		disconnect the 3-planes. By real analyticity, each nonempty
		$P_\pm\cap B_{K/2}$ is a multiplicity-one component of $T_\infty$.
		Since $T_\infty$ agrees with the two planes outside the unit tubes,
		translation invariance implies that any
		remaining component meeting $B_{K/2}$ has a compactly supported closed
		calibrated slice transverse to $\vec v$. This slice is zero by the monotonicity formula.
		Thus $T_\infty=[P_-]+[P_+]$ in $B_{K/2}$.
		
		On each $P_\pm$, the part of $\overline B_4$ at distance at least
		$1/(2K)$ from the axes is a compact subset of the multiplicity-one
		regular set of the limit. Allard regularity \cite[Section~3]{Minter}
		and interior estimates give the required $C^2$ bounds for graphs of
		$T_i$ over $P_{\pm,i}$. Uniqueness of normal projection identifies them
		with the original graphs on overlaps, contradicting the choice of sequence.
		The same argument applies when the affine planes are replaced by sheets
		converging to them in $C^2$.
	\end{proof}
	
	\begin{remark*}
		For a single model plane, the same argument gives extension over
		its intersection with $B_4$, with the same bound.
	\end{remark*}
	
	\subsection{The nonlinear three-annulus estimate}
	
	We return to $L$ and its comparison cone $\mathbf C_s$.
	Given $q\in\mathcal R_s$ with $2/3\le |q|\le7/3$,
	write $V(t,y)$ for the normal vector field representing $L$ as a graph
	over the corresponding smooth branch of $\mathbf C_s$, in coordinates
	centred at $q$.
	Here $t$ is arclength along the ray, with $t=0$ at $q$, and the
	transverse coordinates $y$ are orthonormal there. Recall $U_r$ from
	Lemma~\ref{lem:linear-three-annulus} and put
	\[
	\widetilde U_r=\{(t,y):|t|<2Kr,\ r/2<|y|<2Kr\}.
	\]
	Define the norms
	\[
	\|V\|_r^2=r^{-5}\int_{U_r}|V|^2,\qquad
	\|V\|_{C^2_r}=\sum_{j=0}^2r^j\|\nabla^jV\|_{L^\infty(\widetilde U_r)}.
	\]
	The integral uses the induced measure on this cone branch.
	
	Set
	\[
	\mathcal E(q,r)=r^{-3}\int_{L\cap B_{Cr}(q)}|x^\perp|^2\,d\|L\|,
	\]
	where $C$ is fixed and large enough that $B_{Cr}(q)$ contains the graphs over
	$\widetilde U_{r/K}\cup\widetilde U_r\cup\widetilde U_{Kr}$.
	For sufficiently small $r$, monotonicity gives $\mathcal E(q,r)\le Cr^{-3}E$.
	\begin{proposition}\label{prop:perturbative-transfer}
		Fix $3<\gamma<4$ and choose $K$ sufficiently large as in
		Lemma~\ref{lem:linear-three-annulus}.
		There is $\eta>0$ such that, for sufficiently small $r>0$, if the graph is
		defined on $\widetilde U_{r/K}\cup\widetilde U_r\cup\widetilde U_{Kr}$, then
		\[
		\|V\|_r^2\ge K^3\|V\|_{K r}^2
		\quad\Longrightarrow\quad
		\|V\|_{r/K}^2\ge 2K^{\gamma}\|V\|_r^2,
		\]
		provided $\rho^{-1}\|V\|_{C^2_\rho}\le\eta$ for
		$\rho\in\{r/K,r,Kr\}$ and
		$\mathcal E(q,r)\le\eta\|V\|_r^2$.
		Constants are uniform over the fixed compact model family.
	\end{proposition}
	
	\begin{proof}
		Suppose the assertion fails at centres $q_i$ as $\eta$ and the local radius tend to zero.
		Set
		\[
		v_i(t,y)=\frac{V_i(r_it,r_iy)}{r_i\|V_i\|_{r_i}}.
		\]
		The hypothesis and the failure of the conclusion give
		\[
		\|v_i\|_K^2\le K^{-3},\qquad
		\|v_i\|_1^2=1,\qquad
		\|v_i\|_{1/K}^2<2K^{\gamma}.
		\]
		After rescaling, the corresponding branch of the comparison cone converges
		smoothly to a 3-plane.
		The height bounds and (\ref{eqn:radial-graph-equation}) give
		\[
		\int_{U_\sigma}|\partial_t v_i|^2
		\le C\left(r_i^2+\frac{\mathcal E(q_i,r_i)}{\|V_i\|_{r_i}^2}\right)
		\longrightarrow0,
		\qquad \sigma\in\{K^{-1},1,K\}.
		\]
		
		By interior regularity of the special Lagrangian graph equation,
		the three height bounds give local smooth convergence of the normalized
		fields to $v$. The derivative estimate makes $v$ independent of $t$ and
		gives strong $L^2$ convergence on $U_1$, hence $\|v\|_1^2=1$.
		
		Lower semicontinuity gives
		\[
		\|v\|_K^2\le K^{-3},\qquad
		\|v\|_{1/K}^2\le2K^{\gamma}.
		\]
		The one-form associated with $v$ is closed and coclosed and has the
		same norms. Since $\|v\|_1^2=1$, Lemma~\ref{lem:linear-three-annulus}
		gives $\|v\|_{1/K}^2>2K^{\gamma}$, contradicting the second bound.
	\end{proof}
	
	\subsection{Graphical extension up to the stopping radii}
	\label{sect:stopping-construction}
	
	We shall extend the normal graphs towards $\mathcal R_s$ by successively
	replacing the inner radius $r$ in the domains $U_r$ by $r/K$.
	For $R_0\in[2/3,7/3]$, we consider the domains $U_r\subset\widetilde U_r$
	centred at $R_0q_\pm$ on each local sheet of $\mathbf C_s$.
	
	Set
	\[
	D_{R_0}(r)^2=r^{-5}\int_{L\cap\{|t-R_0|<r,\ r<|y|<Kr\}}
	\dist(x,\mathbf C_s)^2\,d\|L\|.
	\]
	In this integral, $(t,y)$ are the coordinates of the normal projection
	onto each sheet, with $t$ measured from the origin along its ray;
	the integral is summed over the two sheets.
	
	Fix $3<\gamma<4$ and choose $K$ sufficiently large.
	Choose $\eta$ and $\varepsilon_0$ sufficiently small as in
	Proposition~\ref{prop:perturbative-transfer} and
	Lemma~\ref{lem:qualitative-continuation}, respectively.
	Then choose $r_0$ sufficiently small depending on these constants
	and the fixed family of nearby cones.
	Finally, choose $\delta$ in (\ref{eqn:initial-closeness}) sufficiently
	small that Allard regularity supplies the initial graphs with
	rescaled $C^2$ norm at most $\eta$ and
	$D_{R_0}(r_0)^2,D_{R_0}(Kr_0)^2<\varepsilon_0^2$.
	Put $r_j=K^{-j}r_0$ for $j\ge0$. At step $j$, the graphs are already defined
	on $\widetilde U_{r_j}\cup\widetilde U_{Kr_j}$. Apply the following
	stopping criteria in order:
	\begin{enumerate}
		\item \emph{Energy criterion.} Stop if
		$\mathcal E(R_0q_\pm,r_j)\ge\varepsilon_0^2$ for at least one of the two signs.
		\item \emph{Height criterion.} Stop if $D_{R_0}(r_j)^2\ge\varepsilon_0^2$.
	\end{enumerate}
	If neither holds, extend the graphs to $\widetilde U_{r_{j+1}}$
	and repeat with $j+1$.
	At centres where neither criterion is ever met, continue at every scale.
	
	\begin{proposition}[Inductive graphical extension]\label{prop:inductive-extension}
		With these choices, the construction is
		well defined at every centre.
		At every step $j$, $L$ is given by normal
		graphs over the corresponding domains of $\mathbf C_s$, with
		\[
		r_j^{-1}\|V\|_{C^2_{r_j}}\le\eta
		\]
		on both sheets.
	\end{proposition}
	
	\begin{proof}
		The initial graphs and their bounds follow from the choice of $\delta$.
		Suppose the graphs have been constructed down to $r_j$ and neither
		stopping criterion holds. After rescaling by $r_j^{-1}$, the fields
		are defined on $\widetilde U_1$ with $C^2$ norm at most $\eta<1$.
		Uniform transversality makes normal height comparable to distance from
		$\mathbf C_s$ on $U_{r_j}$, so on each sheet
		\[
		\|V\|_{r_j}\le C D_{R_0}(r_j)<C\varepsilon_0.
		\]
		This verifies the $L^2$ hypothesis in
		Lemma~\ref{lem:qualitative-continuation}.
		
		At either centre $q=R_0q_\pm$, put $\vec v=q/|q|$.
		The monotonicity formula gives the required rescaled mass bound, and
		$|q|\vec v^\perp=x^\perp-(x-q)^\perp$ gives
		\[
		r_j^{-3}\int_{L\cap B_{Cr_j}(q)}|\vec v^\perp|^2\,d\|L\|
		\le C\bigl(\mathcal E(q,r_j)+r_j^2\bigr)
		\le C(\varepsilon_0^2+r_0^2).
		\]
		Up to a negligible error of order $O(r_j)$ in the rescaled $C^2$ norm,
		we can replace each local sheet of $\mathbf C_s$ by its affine tangent plane
		at $R_0q_\pm$.
		If both sheets meet $B_{Cr_j}(q)$, their centres are $O(r_j)$ apart,
		so their radial directions differ by $O(r_j)$.
		Hence, up to $\mathrm{SU}(3)$ rotations of order $O(r_j)$, we can
		assume that both radial axes are parallel to $\vec v$.
		Since the centres lie in $\mathbb C\vec v$, the separation of these axes
		lies in $\mathbb R J\vec v$.
		This gives the planes required in Lemma~\ref{lem:qualitative-continuation}.
		Apply Lemma~\ref{lem:qualitative-continuation} with output bound $\eta/(2K)$
		(using its single-plane version if only one sheet meets the ball).
		Since
		\[
		\widetilde U_{r_{j+1}}
		=\{|t|<2r_j,\ r_j/(2K)<|y|<2r_j\}
		\subset B_{4r_j}\cap\{|y|>r_j/(2K)\},
		\]
		Lemma~\ref{lem:qualitative-continuation} supplies the graphs near
		$\overline{\widetilde U_{r_{j+1}}}$ with rescaled $C^2$ norm at most
		$\eta/2$ at scale $r_{j+1}$.
		Returning to the comparison cone adds $O(Kr_j)$, absorbed by the
		choice of $r_0$, and gives
		$r_{j+1}^{-1}\|V\|_{C^2_{r_{j+1}}}\le\eta$.
		This completes the induction.
	\end{proof}
	
	At a stopping radius $r_j$, call
	$B_{Cr_j}(R_0q_\pm)$ the terminal balls.
	
	\subsection{A lower bound for a stopping radius}
	
	For the normal vector fields from the stopping construction, put
	\[
	Q=\sup_{V,j}r_j^{\gamma}\|V\|_{r_j}^2,
	\]
	where $r_j$ ranges over all scales at which $V$ is defined by the
	stopping construction, including stopping radii. Consequently,
	$D_R(r)^2\le CQr^{-\gamma}$
	for every $R$ and every scale $r$ at which the graphs are defined.
	
	\begin{proposition}\label{prop:stopping}
		There are $\delta,\varepsilon_1>0$ and $C>1$ with the following property. Suppose
		(\ref{eqn:initial-closeness}) holds, $D>0$ and $E\le\varepsilon_1 D^2$.
		Then $Q\ge C^{-1}D^2$. At some regular radius $R_0\in[2/3,7/3]$,
		the construction stops by the height criterion at a radius $\rho$ with
		\[
		C^{-1}Q\le\rho^{\gamma}\le CQ.
		\]
		Every stopping radius satisfying the height criterion is at most $C\rho$.
	\end{proposition}
	
	\begin{proof}
		\noindent\emph{Step 1: bound $Q$ from below.}
		Put $f(x)=\dist(x,\mathbf C_s)$.
		Choose a fixed Lipschitz function $\chi$, supported in $[1/4,4]$, and
		a constant $c_0>0$ such that
		\[
		R\chi'(R)+5\chi(R)=1-c_0\mathbf1_{(1,2)}(R),
		\qquad 1/4<R<4.
		\]
		Since $f^2$ is homogeneous of degree two and $|D(f^2)|\le2f$,
		stationarity tested with $\chi(R)f^2x$ gives
		\[
		D^2\le c_0\int_{L\cap\{1<R<2\}}f^2\,d\|L\|+CD\sqrt E.
		\]
		Fix a sufficiently small $0<\tau<1/8$.
		Lemma~\ref{lem:nonconcentration} and $E\le\varepsilon_1D^2$ imply,
		after decreasing $\delta$ and $\varepsilon_1$, that
		\[
		\int_{L\cap\{1<R<2,\ \dist(x,\mathcal R_s)<\tau\}}f^2\,d\|L\|
		\ge C^{-1}D^2.
		\]
		
		Cover the set above, up to a subset of $\mathcal R_s$ of zero
		$\|L\|$-measure, by the images of $U_{r_j}$ under the normal graph
		parametrizations and by terminal balls.
		Choosing $C^{-1}r_j$-separated centres along each ray gives at most
		$C/r_j$ regions at scale $r_j$, with bounded overlap of the associated
		balls. For $r_0$ small, all these balls lie in
		$B_4\setminus\overline B_{1/4}$.
		
		Each graphical region contributes at most $CQ r_j^{5-\gamma}$ to the
		integral of $f^2$. For a terminal ball $B_{Cr_j}(q)$, monotonicity and
		$q\in\mathbf C_s$ give
		\[
		\int_{B_{Cr_j}(q)}f^2\,d\|L\|
		\le Cr_j^2\|L\|(B_{Cr_j}(q))\le Cr_j^5.
		\]
		At a height stopping radius, $r_j^\gamma\le CQ$, so this is also bounded
		by $CQ r_j^{5-\gamma}$.
		For balls satisfying the energy criterion, bounded overlap gives a total
		contribution at most $Cr_j^2E$ at each scale.
		Summing over the scales $r=r_0,K^{-1}r_0,\ldots$ and using $\gamma<4$ gives
		\[
		D^2\le CQ\sum_{\text{scales }r}r^{4-\gamma}
		+CE\sum_{\text{scales }r}r^2
		\le C(Q+E).
		\]
		Taking $\varepsilon_1$ sufficiently small and absorbing the $E$ term,
		we obtain $Q\ge C^{-1}D^2$.
		
		\smallskip
		\noindent\emph{Step 2: select a height stopping radius.}
		Choose a centre $q\in\mathcal R_s$, with $R_0=|q|$, and a radius
		$\rho$ at which the corresponding field $V$ is defined, such that
		\[
		Q/2<\rho^\gamma\|V\|_\rho^2\le C\rho^\gamma,
		\]
		where the upper bound follows from the uniform graphical estimate.
		(By a small perturbation of the centre, we may assume that the slice
		at $R_0$ is regular.)
		Since $Q\ge C^{-1}D^2$, Lemma~\ref{lem:nonconcentration} gives
		$r_0^\gamma\|V\|_{r_0}^2<Q/2$ at every centre once $\delta$ and
		$\varepsilon_1$ are sufficiently small.
		Thus $\rho<r_0$, so the graphs are also defined at scale $K\rho$.
		Since $\gamma>3$ and $K$ is sufficiently large, the definition of $Q$ gives
		\[
		\|V\|_\rho^2>\tfrac12K^\gamma\|V\|_{K\rho}^2
		\ge K^3\|V\|_{K\rho}^2.
		\]
		Moreover, $E\le\varepsilon_1D^2$ and $Q\ge C^{-1}D^2$ give
		\[
		\frac{E}{\rho^3\|V\|_\rho^2}
		\le C\frac{E}{Q}\rho^{\gamma-3}
		\le C\varepsilon_1\rho^{\gamma-3}.
		\]
		This verifies the normalized energy bound in
		Proposition~\ref{prop:perturbative-transfer} for sufficiently small $\varepsilon_1$.
		Also, $E/\rho^3\le C\varepsilon_1Q/\rho^3
		\le C\varepsilon_1\rho^{\gamma-3}$, so the energy criterion fails
		at every centre and every radius at least $\rho$.
		
		If the height criterion failed at $R_0$ and $\rho$, the construction
		would reach $\rho/K$, and Proposition~\ref{prop:perturbative-transfer} would give
		\[
		(\rho/K)^\gamma\|V\|_{\rho/K}^2
		\ge2\rho^\gamma\|V\|_\rho^2>Q,
		\]
		contrary to the definition of $Q$.
		Thus $\rho$ is a height stopping radius.
		At any radius $r$ satisfying the height criterion,
		\[
		\varepsilon_0^2r^\gamma
		\le r^\gamma D_R(r)^2\le CQ.
		\]
		Together with $Q\le C\rho^\gamma$, this proves
		$C^{-1}Q\le\rho^\gamma\le CQ$ and $r\le C\rho$.
	\end{proof}
	
	\subsection{Estimates on the graphical region}
	\label{sect:graphical-radial-slices}
	
	The Legendrian neighbourhood construction, applied along $F_s$, gives local
	diffeomorphisms $\Psi_s$ from a neighbourhood of the zero section of
	$T^*\Sigma\times\R$ into $S^5$, depending smoothly on $s$, such that
	\[
	\Psi_s(p,0,0)=F_s(p),\qquad
	\Psi_s^*\lambda=2\left(dz-\sum_j\xi_j\,dp^j\right).
	\]
	Here $(p^j,\xi_j)$ are local cotangent coordinates and $z$ is the
	coordinate on the $\R$-factor.
	Choose these coordinates so that the infinitesimal displacement in
	the fibre directions is $J\xi^\sharp+2zJF_s$.
	Let $\Gamma_R=(\lambda_{0,R})_\#\langle L,|x|,R\rangle$ be the radial
	slice normalized to $S^5$, and write its graphical part as
	$\Psi_s(p,\alpha_R(p),u_R(p))$, where $\alpha_R$ is a one-form
	and $u_R$ is a real-valued function.
	The graphical estimates of
	Proposition~\ref{prop:inductive-extension} apply in these coordinates.
	
	Recall the stopping radius $\rho$ from Proposition~\ref{prop:stopping}.
	Fix $1<p<2$ and put $\rho_*=\rho^p$.
	Assume $E\le\rho_*^5$ and take $\rho$ sufficiently small.
	The construction then reaches $\rho_*$ wherever the height criterion
	has not been met.
	Fix any radial value $R\in[2/3,7/3]$ and a puncture $p_\pm\in\mathcal P$
	corresponding to $q_\pm$.
	Let $r_j$ be the stopping radius at $Rq_\pm$, taken to be zero if
	the construction never stops, and put
	\[
	\sigma=\max\{r_j,\rho_*\}.
	\]
	Let $\bar u_R$ be the circular average of $u_R$ around $p_\pm$.
	The norms $C^k_r$ below are computed after rescaling each annulus
	to unit size.
	
	\begin{proposition}[Estimates for the potential]\label{prop:exterior-scalar}
		There are $\delta,\rho_0,C>0$ such that, if
		(\ref{eqn:initial-closeness}) holds, $0<\rho<\rho_0$ and $E\le\rho_*^5$, then
		$\|u_R\|_{C^2(\Sigma_{r_0})}\le C\rho^2$, where
		$\Sigma_{r_0}=\Sigma\setminus\bigcup_{p\in\mathcal P}B_{r_0}^{\Sigma}(p)$.
		In geodesic polar coordinates $(r,\theta)$ about $p_\pm$, for
		$\sigma<r<r_0$,
		\begin{gather}
			|u_R(r,\theta)|\le C\rho^2(1+\log(r_0/r)),\qquad
			\|u_R-\bar u_R\|_{C^2_r}\le C\rho^2,\notag\\
			|du_R|+r|D^2u_R|
			\le C\left(\frac{\sigma^2}{r}+\rho^2\right).
			\label{eqn:potential-estimates}
		\end{gather}
		For every $\varepsilon>0$, there is $C(\varepsilon)>1$ such that,
		after decreasing $\delta,\rho_0$ if necessary, there is $a\in\R$ with
		$|a|\le C\sigma^2$ such that
		\begin{equation}\label{eqn:neck-coefficient}
			|r\bar u_R'(r)-a|
			\le\varepsilon\sigma^2+C\rho^2r,
			\qquad C(\varepsilon)\sigma<r<r_0.
		\end{equation}
		The constants are uniform in $R\in[2/3,7/3]$ and independent of
		$\rho$ and $\sigma$. We write $C(\varepsilon)$ for constants that may
		depend on $\varepsilon$, reserving $C$ for constants independent of it.
	\end{proposition}
	
	\begin{proof}
		\emph{Step 1: estimate the radial error.}
		Write $Y_R(p)=\Psi_s(p,\alpha_R(p),u_R(p))$ and let $d$ denote
		differentiation on $\Sigma$.
		Under the parametrization $(R,p)\mapsto RY_R(p)$ of $L$, the
		Liouville form pulls back to
		\[
		2R^2(\partial_Ru_R\,dR+du_R-\alpha_R).
		\]
		Its closedness on $L$ gives $d\alpha_R=0$ and
		$R\partial_R\alpha_R=2(du_R-\alpha_R)$.
		On $\Gamma_R$,
		\begin{equation}\label{eqn:contact-equation}
			Y_R^*\lambda=2(du_R-\alpha_R).
		\end{equation}
		The restriction of $\lambda$ to the three-dimensional regular part
		of $L$ is harmonic and satisfies $|\lambda|=|x^\perp|$.
		Wherever Proposition~\ref{prop:inductive-extension} supplies the graphs,
		interior estimates on $U_r$ give
		\begin{equation}\label{eqn:liouville-interior}
			r^j\|\nabla^j\lambda\|_{C^1_r}
			\le C r^{-3/2}\sqrt E,\qquad j=0,1.
		\end{equation}
		Since $2/3\le R\le7/3$, the pullback formula and the identity
		for $\partial_R\alpha_R$ give
		\[
		\|\partial_Ru_R\|_{C^2_r}
		+\|\partial_R\alpha_R\|_{C^2_r}
		+\|\partial_RY_R\|_{C^2_r}
		\le Cr^{-3/2}\sqrt E.
		\]
		
		To derive an equation for $u_R$ on each radial slice, use
		(\ref{eqn:contact-equation}) to substitute $\alpha_R=du_R-\tfrac12Y_R^*\lambda$
		in the special Lagrangian equation for $L$, and collect the Liouville and
		radial terms in an error term $f$, controlled by (\ref{eqn:liouville-interior}).
		Let $D^2u_R$ denote the covariant Hessian for the metric
		induced by $F_s$. The special Legendrian operator for the graph
		$\Psi_s(p,du_R,u_R)$ is at most quadratic in $D^2u_R$ and at
		$u_R=du_R=0$ equals $\operatorname{tr}D^2u_R$.
		Its linearization is $\operatorname{tr}D^2u_R+6u_R$, so the special
		Lagrangian equation has the schematic form
		\[
		\begin{split}
			\operatorname{tr}D^2u_R+6u_R
			={}&(u_R,du_R)\#(u_R,du_R)+(u_R,du_R)\#D^2u_R\\
			&+(u_R,du_R)\#D^2u_R\#D^2u_R+f.
		\end{split}
		\]
		Here $\#$ denotes sums of contractions with smooth coefficients depending
		on $(p,u_R,du_R)$; these include the curvature of $\Lambda_s$ and the
		expression of $\Omega$ in the contact coordinates.
		The graphical bounds in Proposition~\ref{prop:inductive-extension} give
		\[
		\|u_R\|_{C^1_r}+\|du_R\|_{C^1_r}\le Cr,\qquad
		\|D^2u_R\|_{C^1_r}\le C.
		\]
		In geodesic normal coordinates about the puncture in $\Sigma$, absorb the
		nonlinear terms involving $D^2u_R$ into $a^{ij}$ to obtain
		\begin{equation}\label{eqn:potential-equation}
			a^{ij}\partial_i\partial_j u_R+b^i\partial_i u_R+cu_R=f.
		\end{equation}
		Here $c$ is uniformly bounded, and $E\le\rho_*^5\le\sigma^5$ gives
		\begin{equation}\label{eqn:potential-equation-bounds}
			\begin{split}
				\|f\|_{C^1_r}&\le Cr^{-5/2}\sqrt E
				\le C(\sigma/r)^{5/2},\\
				\|a^{ij}-\delta^{ij}\|_{C^1_r}+\|b^i\|_{C^1_r}
				&\le Cr+Cr^{-3/2}\sqrt E\le Cr.
			\end{split}
		\end{equation}
		
		\smallskip
		\noindent\emph{Step 2: local estimates in terms of the exterior norm.}
		Set $M=\sup_{2/3\le R\le7/3}\|u_R\|_{C^2(\Sigma_{r_0})}$, and let
		$\Delta_0=\partial_1^2+\partial_2^2$ be the Euclidean Laplacian in the normal
		coordinates. The graphical bounds and (\ref{eqn:contact-equation}) give
		$\|u_R-\bar u_R\|_{C^2_\sigma}+|\sigma\bar u_R'(\sigma)|\le C\sigma^2$,
		while the boundary data at $r=r_0$ are bounded by $M$.
		For $\sigma<r_0/4$, define the finite quantity
		\[
		A=1+\sup_{\sigma<r<r_0}
		\frac{\|du_R\|_{C^1_r}}{\sigma^2/r+M/r_0}.
		\]
		Integrating along radial lines from $r_0$ gives
		$|u_R|\le CA\{M+\sigma^2(1+\log(r_0/r))\}$.
		Hence (\ref{eqn:potential-equation})--(\ref{eqn:potential-equation-bounds}) give
		\[
		|\Delta_0u_R-f|
		\le C(r|D^2u_R|+r|du_R|+|u_R|)
		\le CA\left(\frac{\sigma^2}{r}+\frac{M}{r_0}\right).
		\]
		The circular average satisfies
		\begin{equation}\label{eqn:circular-average}
			r\bar u_R'(r)-\sigma\bar u_R'(\sigma)
			=\int_\sigma^r t\,\overline{\Delta_0u_R}(t)\,dt.
		\end{equation}
		In the variable $t=\log r$, the Dirichlet Green function for each Fourier
		mode $k\ne0$ is bounded by
		$(2|k|)^{-1}e^{-|k||t-s|}$. The boundary data and $f$ contribute
		$C(\sigma^2+rM/r_0)$ and $C\sqrt E\,r^{-1/2}\le C\sigma^2$, respectively.
		Together with the bound for $\Delta_0u_R-f$, this gives
		\[
		|u_R-\bar u_R|
		\le C\left(\sigma^2+\frac{r}{r_0}M\right)
		+CAr\{\sigma^2(1+\log(r_0/r))+M\}.
		\]
		Equation (\ref{eqn:circular-average}) gives the same bound for $r|\bar u_R'|$, using
		$\int_\sigma^{r_0}t(\sigma/t)^{5/2}\,dt\le2\sigma^2$.
		Interior elliptic estimates for (\ref{eqn:potential-equation}) on slightly larger annuli give
		\[
		A\le C+CA\sup_{\sigma<r<r_0}r(1+\log(r_0/r))\le C+Cr_0A,
		\]
		with $C$ independent of $\sigma$ and small $r_0$.
		Fix $r_0$ independently of $\sigma$ so that $Cr_0\le1/2$; absorbing the last term gives $A\le2C$.
		Thus $\|u_R-\bar u_R\|_{C^2_r}\le C(M+\sigma^2)$ and
		$|du_R|+r|D^2u_R|\le C(\sigma^2/r+M)$.
		
		\smallskip
		\noindent\emph{Step 3: bound the exterior norm.}
		Suppose the bound fails along a sequence with $\delta_i,\rho_i\to0$ and
		$M_i/\rho_i^2\to\infty$. Set $f_i=\dist(\cdot,\mathbf C_{s_i})$.
		By monotonicity, the $L^2$ norm of $x^\perp/M_i$ on the full annulus is
		$O(\sqrt{E_i}/M_i)=o(1)$.
		Since $\sigma_i\le C\rho_i$, Steps 1--2 and
		(\ref{eqn:radial-graph-equation}) give, after passing to a subsequence,
		smooth convergence of $V_i/M_i$ on compact regular regions to a
		homogeneous Jacobi field $J\nabla(R^2v)$ with
		$\|v\|_{C^2(\Sigma_{r_0})}=1$.
		Step 2 gives $|dv|\le C$ near the punctures, so
		Lemma~\ref{lem:jacobi} extends $v$ smoothly across them.
		
		On $1<R<2$, the tube of radius $C\rho_i$ contributes
		$O(\rho_i^3)=o(M_i)$ to $\int f_i$, and Step 2 gives uniform
		integrability of $f_i/M_i$ elsewhere. Together with the vanishing radial
		error above, these bounds give strong $L^1$ convergence of $f_i/M_i$
		to $|J\nabla(R^2v)|$ on $B_4\setminus\overline B_{1/4}$.
		
		By integrability, choose cone variations $\mathbf C_{s_i(t)}$ whose
		normal velocities $W_i$ at $t=0$ converge smoothly to $J\nabla(R^2v)$.
		The minimising property of $\mathbf C_{s_i}$ gives
		\[
		\begin{aligned}
			0&=\left.\frac{d}{dt}\right|_{t=0}
			\int_{L_i\cap(B_4\setminus\overline B_{1/4})}
			\dist(x,\mathbf C_{s_i(t)})^2\,d\|L_i\|\\
			&=-2\int_{L_i\cap(B_4\setminus\overline B_{1/4})}
			\langle x-y_i(x),W_i(y_i(x))\rangle\,d\|L_i\|,
		\end{aligned}
		\]
		where $y_i(x)\in\mathbf C_{s_i}$ is a nearest point to $x$.
		After division by $M_i$, the integrand converges in $C^0$ on compact
		regular regions to $|J\nabla(R^2v)|^2$ and is bounded by $Cf_i/M_i$
		near the singular rays.
		Strong $L^1$ convergence therefore gives, after cancelling the radial factor,
		\[
		\int_\Sigma(4v^2+|dv|^2)\,d\mathrm{Area}=0.
		\]
		This contradicts the normalization of $v$ and proves $M\le C\rho^2$.
		Step 2 now gives the stated bounds for $u_R$ and its derivatives.
		
		Finally, take $a=C(\varepsilon)\sigma\bar u_R'(C(\varepsilon)\sigma)$
		for $C(\varepsilon)$ sufficiently large.
		The derivative bound gives $|a|\le C\sigma^2$, since $p<2$ and
		$\rho^2/\sigma\le\rho^{2-p}\ll1$.
		Substituting (\ref{eqn:potential-equation-bounds}) and the bounds just proved into
		(\ref{eqn:circular-average}) gives
		\[
		\begin{split}
			|r\bar u_R'(r)-a|
			&\le\int_{C(\varepsilon)\sigma}^r
			t\,|\overline{\Delta_0u_R}(t)|\,dt\\
			&\le C\rho^2r+\varepsilon\sigma^2,
		\end{split}
		\]
		after choosing $C(\varepsilon)$ large and then $\rho_0$ small.
	\end{proof}

	\subsection{Lawlor neck extraction}
	\label{sect:selected-neck}\label{sect:neck-estimates-stopping-radii}
	
	\begin{lemma}\label{lem:common-slice}
		Fix $0<\varepsilon<1$. There are $C(\varepsilon)>1$ and $\delta,\rho_0>0$
		with the following property. Assume (\ref{eqn:initial-closeness}) holds,
		$0<\rho<\rho_0$ and $E\le\rho_*^5$.
		For every regular $R_0\in[2/3,7/3]$ and every stopping radius
		$\sigma>C(\varepsilon)\rho_*$,
		$L$ in $\bigcup_\pm B_{C(\varepsilon)\sigma}(R_0q_\pm)$ is a multiplicity-one
		normal graph $V$ over a translated product of a Lawlor neck and a line, with
		$\|V\|_{C^2_\sigma}\le\varepsilon\sigma$.
		On the normalized slice $\Gamma_{R_0}$, the logarithmic coefficients
		in (\ref{eqn:neck-coefficient}) may be chosen as $\pm a$ on the
		corresponding $\pm$ ends, where
		\[
		C^{-1}\sigma^2\le |a|\le C\sigma^2.
		\]
		Here the constant $C$ is independent of $\varepsilon$.
	\end{lemma}
	
	\begin{proof}
		\emph{Step 1: extract a nontrivial neck.}
		Since
		\[
		\mathcal E(R_0q_\pm,\sigma)
		\le C\sigma^{-3}E\le C\sigma^2<\varepsilon_0^2,
		\]
		the energy criterion fails. The height criterion gives
		\[
		\int_{L\cap\{|t-R_0|<\sigma,\ \sigma<|y|<K\sigma\}}
		\dist(x,\mathbf C_s)^2\,d\|L\|(x)
		\ge\varepsilon_0^2\sigma^5.
		\]
		Consequently $\int_{U_\sigma}|V|^2\ge C^{-1}\varepsilon_0^2\sigma^5$
		on at least one sheet.
		At the corresponding centre $q\in\{R_0q_-,R_0q_+\}$,
		Proposition~\ref{prop:exterior-scalar} and interior estimates give
		\[
		|V|+r|\nabla V|+r^2|\nabla^2V|\le C\sigma^2/r
		\]
		on the exterior graphical region $\sigma\le r\le C_2(\varepsilon)\sigma$,
		for $C_2(\varepsilon)$ as in Lemma~\ref{lem:effective-lawlor}
		and sufficiently small $\rho_0$.
		After rescaling by $\sigma^{-1}$, this is the graphical bound in
		Lemma~\ref{lem:effective-lawlor}.
		The monotonicity formula gives
		\[
		\|L\|(B_{C_2(\varepsilon)\sigma}(q))\le C(C_2(\varepsilon)\sigma)^3.
		\]
		
		Consider the affine tangent planes to the local sheets of $\mathbf C_s$
		at $R_0q_\pm$.
		The centres satisfy $R_0|q_+-q_-|\le C\sigma$. Otherwise the single-plane
		version in the remark after Lemma~\ref{lem:effective-lawlor}, applied
		at scale $\sigma$, would give a normal graph over the local sheet through
		$q$ with $\|V\|_{C^2_\sigma}\le c\varepsilon_0\sigma$, contradicting
		the lower bound on $\int_{U_\sigma}|V|^2$ for sufficiently small $c$.
		
		The radial directions of the two planes therefore differ by $O(\sigma)$.
		An $O(\sigma)$ special unitary rotation of one of the two planes about
		its centre makes both parallel to $\vec v=q/|q|$.
		In $B_{C_2(\varepsilon)\sigma}(q)$, the local sheets of $\mathbf C_s$ are then
		$C^4_\sigma$ graphs of norm $O(C(\varepsilon)\sigma^2)$ over these planes.
		Since $q^\perp=x^\perp-(x-q)^\perp$, the energy and mass bounds give
		\[
		\sigma^{-3}\int_{L\cap B_{C_2(\varepsilon)\sigma}(q)}
		|\vec v^\perp|^2\,d\|L\|
		\le C(\varepsilon)(\sigma^{-3}E+\sigma^2)\le C(\varepsilon)\sigma^2.
		\]
		The lower bound on $\int_{U_\sigma}|V|^2$ verifies the remaining
		height hypothesis of Lemma~\ref{lem:effective-lawlor} after rescaling
		by $\sigma^{-1}$, for suitable choices of $C_1$ and $\rho_0$.
		Lemma~\ref{lem:effective-lawlor} therefore gives
		the claimed normal graph on a ball containing both terminal balls,
		with neck parameter of modulus comparable to $\sigma^2$.
		
		\smallskip
		\noindent\emph{Step 2: identify the coefficients.}
		To identify the coefficients of $u_{R_0}$, work on $\Gamma_{R_0}$
		and apply the dilation $\lambda_{0,R_0}$ to the Lawlor neck model.
		Choose the model coordinates so that the logarithmic term at the plus end
		is $a\log r$, and write its equation as $Z_1Z_2=a+ib$. Then
		\[
		\int_{|Z_1|=r}\lambda=4\pi b.
		\]
		Comparing this with the integral over the corresponding loop on
		$\Gamma_{R_0}$ and using
		(\ref{eqn:liouville-interior}) gives
		\[
		|b|\le C\varepsilon\sigma^2+C\sigma^{-1/2}\sqrt E
		\le C\left(\varepsilon+
		\left(\frac{\rho_*}{\sigma}\right)^{5/2}\right)\sigma^2.
		\]
		Choosing a sufficiently accurate approximation and a sufficiently large
		cutoff constant in $\sigma>C(\varepsilon)\rho_*$ gives $|a|\asymp\sigma^2$.
		The model potentials have logarithmic terms $\pm a\log r$
		on the corresponding $\pm$ ends. The neck approximation at $r=C(\varepsilon)\sigma$
		identifies their coefficients with those in
		(\ref{eqn:neck-coefficient}) to accuracy $\varepsilon\sigma^2$.
	\end{proof}
	
	\paragraph{Flux interpretation of $a$.}
	For the points $q_-,q_+\in\Lambda_s$ corresponding to each double point
	of $\Lambda$, put
	\[
	P=q_-q_-^*=q_+q_+^*,\qquad
	\eta_H=\iota_{JHx}\iota_x\Ree\Omega
	\]
	for any trace-free Hermitian matrix $H$.
	On $\Gamma_{R_0}$, denote the graphical circles at $r=C(\varepsilon)\sigma$ by
	$c_-,c_+$, oriented as boundaries of the corresponding disks
	$B_{C(\varepsilon)\sigma}^{\Sigma}(p_\pm)$.
	
	\smallskip
	\noindent\emph{Claim 1.} For every regular $R_0\in[2/3,7/3]$ and every
	such pair $q_-,q_+$, the sum of the fluxes through $c_-$ and $c_+$ satisfies
	\begin{equation}\label{eqn:paired-boundary-flux}
		\left|\int_{c_-}\eta_H+\int_{c_+}\eta_H\right|
		\le C\|H\|\sqrt E.
	\end{equation}
	Moving either circle within its graphical annulus changes its flux
	by at most $C\|H\|\sqrt E$.
	
	\begin{proof}
		The three-dimensional graphical collars
		allow us to keep the circles $r=C(\varepsilon)\sigma$ fixed in the link coordinates
		as $R$ varies over an interval $I$ containing $R_0$, with $|I|\asymp\sigma$.
		Write $F_H(R)$ for the sum of their fluxes on $\Gamma_R$.
		Restrict each slice to the cutoff neighbourhood bounded by these circles.
		The corresponding undilated region $S\subset L$ lies in
		$\bigcup_\pm B_{C(\varepsilon)\sigma}(R_0q_\pm)$, so
		$\|L\|(S)\le C(\varepsilon)\sigma^3$.
		The identity
		\begin{equation}\label{eqn:ambient-flux}
			d\eta_H=-3\iota_{JHx}\Ree\Omega
		\end{equation}
		and the special Lagrangian condition give
		\[
		\int_I|F_H(R)|\,dR
		\le C\|H\|\int_S|x^\perp|\,d\|L\|
		\le C(\varepsilon)\|H\|\sigma^{3/2}\sqrt E.
		\]
		On the circles, Cartan's formula and the bound for $\partial_RY_R$
		in Step 1 of Proposition~\ref{prop:exterior-scalar} give
		\[
		|F_H'(R)|
		\le C(\varepsilon)\|H\|\sigma\|\partial_RY_R\|_{C^0}
		\le C(\varepsilon)\|H\|\sigma^{-1/2}\sqrt E.
		\]
		Consequently,
		\[
		|F_H(R_0)|
		\le |I|^{-1}\int_I|F_H(R)|\,dR+|I|\sup_I|F_H'|
		\le C(\varepsilon)\|H\|\sigma^{1/2}\sqrt E
		\le C\|H\|\sqrt E.
		\]
		The last inequality holds after decreasing $\rho_0$ in terms of
		$\varepsilon$, since $\sigma\le C\rho$.
		On a graphical annulus, (\ref{eqn:ambient-flux}) and
		(\ref{eqn:liouville-interior}) also show that moving either circle
		changes its flux by at most $C\|H\|\sqrt E$.
	\end{proof}
	
	\smallskip
	\noindent\emph{Claim 2.} On a neck with $\sigma>C(\varepsilon)\rho_*$, the parameter
	$a\in\R$ can be normalized so that
	\begin{equation}\label{eqn:flux-coefficient}
		\int_{c_+}\eta_{P-I/3}=4\pi a,\qquad
		\int_{c_-}\eta_{P-I/3}=-4\pi a+O(\sqrt E),
	\end{equation}
	while preserving the conclusions of Lemma~\ref{lem:common-slice}.
	For $\sigma\le C(\varepsilon)\rho_*$, the first identity defines $a$ with
	$|a|\le C(\varepsilon)\rho_*^2$.
	If $|q_+-q_-|>C(\varepsilon)\sigma$, then $|a|\le C\sqrt E$.
	
	\begin{proof}
		For $\sigma>C(\varepsilon)\rho_*$, write $\tilde a$ for the real parameter of the
		Lawlor neck model in Lemma~\ref{lem:common-slice}, and define $a$ by the first
		identity in (\ref{eqn:flux-coefficient}).
		Set $H=P-I/3$ and $h=\langle HF_s,F_s\rangle$.
		By (\ref{eqn:contact-equation}) and (\ref{eqn:liouville-interior}),
		the linear part of the flux is
		$3\int(h\partial_\nu u_{R_0}-u_{R_0}\partial_\nu h)\,ds
		+O(r^{-1/2}\sqrt E)$.
		Since $h(p_+)=2/3$ and $dh(p_+)=0$, the leading logarithmic term
		$\tilde a\log r$ in $u_{R_0}$ contributes $3\cdot(2/3)\cdot2\pi\tilde a=4\pi\tilde a$.
		
		By Proposition~\ref{prop:exterior-scalar}, the contribution from the
		average zero part of $u_{R_0}$ is bounded by
		$C(\varepsilon\sigma^2+\rho^2r\log(1/\rho)+r^{-1/2}\sqrt E)$,
		and the quadratic nonlinear error is bounded by
		$C(\sigma^4/r^2+\rho^4\log^2(1/\rho)+E/r^3)$.
		With $C(\varepsilon)$ in $r=C(\varepsilon)\sigma$ and
		$\sigma>C(\varepsilon)\rho_*$ chosen sufficiently large,
		the quadratic error is at most $\varepsilon\sigma^2$
		for sufficiently small $\rho$. Hence
		\[
		|a-\tilde a|
		\le C\varepsilon\sigma^2+C(\varepsilon)\rho^2\sigma\log(1/\rho)
		+C\sigma^{-1/2}\sqrt E
		\le C\varepsilon\sigma^2.
		\]
		Replacing $\tilde a$ by $a$ changes the Lawlor neck model by only
		$O(\varepsilon\sigma)$ in the $C^2_\sigma$ norm. With a sufficiently accurate
		initial approximation, the conclusions of
		Lemma~\ref{lem:common-slice} therefore remain valid.
		In particular, (\ref{eqn:neck-coefficient}) holds with $\pm a$
		on the corresponding $\pm$ ends. Claim 1 gives the second identity in
		(\ref{eqn:flux-coefficient}).
		
		When $\rho_*\le\sigma\le C(\varepsilon)\rho_*$, the coefficient bound in
		Proposition~\ref{prop:exterior-scalar} and $E\le\rho_*^5$ give
		$|a|\le C(\sigma^2+\sigma^{-1/2}\sqrt E)\le C(\varepsilon)\rho_*^2$.
		If $|q_+-q_-|>C(\varepsilon)\sigma$, the two cutoff regions lie in disjoint balls
		of radius $O(C(\varepsilon)\sigma)$ about $q_\pm$.
		For each nearby slice used in Claim 1, their boundaries meet $\Gamma_R$
		only along $c_-(R)$ and $c_+(R)$, respectively.
		Since $\Gamma_R$ is closed, its restriction to either region has boundary
		the corresponding circle. The Stokes and radial estimates in Claim 1
		therefore apply to each circle separately. Taking $H=P-I/3$ gives
		\[
		4\pi|a|=\left|\int_{c_+}\eta_{P-I/3}\right|
		\le C(\varepsilon)\sigma^{1/2}\sqrt E\le C\sqrt E.
		\]
	\end{proof}

	\subsection{The upper bound for the density excess}
	
	Recall $\Gamma_R$ and put
	$\zeta=(\iota_x\Ree\Omega)|_{S^5}$. Since each $\Lambda_s$ is a closed
	minimal immersion, its area is independent of $s$. Thus
	(\ref{eqn:stokes-boundary-integral}) gives
	\[
	3m(R)=\int_{\Gamma_R}\zeta-3\Theta_0
	=\int_{\Gamma_R}\zeta-\int_{\Lambda_s}\zeta.
	\]
	
	\begin{proposition}[Excess estimate]\label{prop:density}
		There are $\delta,C>0$ such that, if (\ref{eqn:initial-closeness}) holds
		and $\mathbf C_s$ is a best approximating cone with
		$s\in\operatorname{int}\mathcal M$, then
		\[
		m(1)\le C(D^2+E).
		\]
	\end{proposition}
	
	\begin{proof}
		We shall prove, for almost every $R\in(5/4,7/4)$, that
		\begin{equation}\label{eqn:excess-slice-bound}
			\left|\int_{\Gamma_R}\zeta-\int_{\Lambda_s}\zeta\right|
			\le C\int_{\Gamma_R}
			\bigl(\dist(x,\mathbf C_s)^2+|\lambda|^2\bigr)\dd\|\Gamma_R\|.
		\end{equation}
		By coarea,
		\[
		\int_{5/4}^{7/4}\int_{\Gamma_R}
		\bigl(\dist(x,\mathbf C_s)^2+|\lambda|^2\bigr)
		\dd\|\Gamma_R\|\,dR\le C(D^2+E).
		\]
		Selecting a good radius $R\in(5/4,7/4)$ and using
		(\ref{eqn:excess-slice-bound}) and monotonicity, we obtain
		$m(1)\le m(R)\le C(D^2+E)$.
		
		We may assume that $\Gamma_R$ is smooth. Put
		$\zeta'=(\iota_x\operatorname{Im}\Omega)|_{S^5}$. The argument uses
		a homotopy $H_t$, $0\le t\le1$, with the following properties on
		$\Gamma_R$:
		\[
		\begin{gathered}
			H_1=\mathrm{id},\qquad (H_0)_\#\Gamma_R=\Lambda_s,\qquad
			H_0^*\zeta'=0,\\
			|\dot H_t(x)|\le C\dist(x,\mathbf C_s),\qquad |dH_t|\le C.
		\end{gathered}
		\]
		The differential bounds hold almost everywhere, with a uniform constant.
		We shall obtain $H_t$ by short geodesic interpolation from a uniformly
		Lipschitz map $H_0$ to the identity; first we prove the estimate from
		these properties.
		
		\smallskip
		\noindent\emph{Step 1: compare the integrals.}
		For a differential form $\vartheta$, write
		\[
		\mathcal K_t\vartheta
		=\int_0^t H_\tau^*(\iota_{\dot H_\tau}\vartheta)\,d\tau.
		\]
		Since $H_0^*\zeta'=0$ on $\Gamma_R$, the homotopy formula gives
		\[
		H_t^*\zeta'=d\mathcal K_t\zeta'+\mathcal K_t(d\zeta'),\qquad
		\partial_t(H_t^*\lambda)
		=d\bigl(\lambda(\dot H_t)\bigr)
		+H_t^*(\iota_{\dot H_t}d\lambda).
		\]
		Applying Stokes to the currents defined by the Lipschitz homotopy,
		and using $d\zeta=-3\lambda\wedge\zeta'$, the two identities above,
		and $\mathcal K_0=0$, we obtain
		\[
		\begin{aligned}
			&\int_{\Gamma_R}\zeta-\int_{\Lambda_s}\zeta\\
			&=-3\int_0^1\int_{\Gamma_R}
			\left\{\lambda(\dot H_t)H_t^*\zeta'
			-H_t^*\lambda\wedge H_t^*(\iota_{\dot H_t}\zeta')\right\}\,dt\\
			&=3\int_0^1\int_{\Gamma_R}
			\Bigl\{\partial_t(H_t^*\lambda\wedge\mathcal K_t\zeta')
			-H_t^*(\iota_{\dot H_t}d\lambda)\wedge\mathcal K_t\zeta'
			-\lambda(\dot H_t)\mathcal K_t(d\zeta')\Bigr\}\,dt\\
			&=3\int_{\Gamma_R}\lambda\wedge\mathcal K_1\zeta'
			-3\int_0^1\int_{\Gamma_R}
			\left\{H_t^*(\iota_{\dot H_t}d\lambda)\wedge\mathcal K_t\zeta'
			+\lambda(\dot H_t)\mathcal K_t(d\zeta')\right\}\,dt.
		\end{aligned}
		\]
		The second equality uses integration by parts on $\Gamma_R$ and the
		product rule in $t$.
		Each factor involving $\dot H_t$ or $\mathcal K_t$ is bounded by
		$C\dist(x,\mathbf C_s)$. The time integral is therefore bounded by
		$C\int_{\Gamma_R}\dist(x,\mathbf C_s)^2\dd\|\Gamma_R\|$, and Young's
		inequality bounds the term at $t=1$ by the right-hand side of
		(\ref{eqn:excess-slice-bound}).
		
		\smallskip
		\noindent\emph{Step 2: construct the homotopy.}
		Away from the points $q_-,q_+$ corresponding to the double points of
		$\Lambda$, we use the usual projection onto the nearby sheet.
		Near $q_-$ and $q_+$, we join the two projections continuously
		by allowing the image to move along the short Reeb segment between
		$q_-$ and $q_+$. When the two points coincide, this segment shrinks to
		their intersection point. We shall construct a uniformly Lipschitz
		retraction $\Pi_s$ from a fixed neighbourhood of $\Lambda_s$ onto this
		union of sheets and segments. Since it fixes $\Lambda_s$, it satisfies
		\[
		|\Pi_s(x)-x|\le C\dist(x,\Lambda_s).
		\]
		
		Let $\beta\ge0$ be the length of the short Reeb segment between
		$q_-$ and $q_+$. Choose local coordinates
		$(x,y,\tau)\in\R^2\times\R^2\times\R$ in which the two local sheets
		of $\Lambda_s$ are the local coordinate planes $\{(x,0,-\beta/2)\}$
		and $\{(0,y,\beta/2)\}$, and the Reeb segment is
		$\{(0,0,\tau):|\tau|\le\beta/2\}$.
		Set $w=\tau+|y|-|x|$ and define
		\[
		\Pi_s(x,y,\tau)=
		\begin{cases}
			\left(\dfrac{-2w-\beta}
			{\max\{\sqrt{|x|^2+|y|^2},-2w-\beta\}}x,0,-\beta/2\right),
			& w<-\beta/2,\\[8pt]
			(0,0,w),& |w|\le\beta/2,\\[4pt]
			\left(0,\dfrac{2w-\beta}
			{\max\{\sqrt{|x|^2+|y|^2},2w-\beta\}}y,\beta/2\right),
			& w>\beta/2.
		\end{cases}
		\]
		The three cases take values in the first plane, the joining segment,
		and the second plane, respectively. In particular,
		\[
		\Pi_s(x,y,\tau)=
		\begin{cases}
			(x,0,-\beta/2),& |x|>3\bigl(|y|+|\tau+\beta/2|\bigr),\\
			(0,y,\beta/2),& |y|>3\bigl(|x|+|\tau-\beta/2|\bigr).
		\end{cases}
		\]
		This is a retraction onto the two planes and their joining segment,
		with Lipschitz constant independent of $\beta$.
		
		Away from fixed neighbourhoods of these points, extend $\Pi_s$
		by a tubular projection onto $\Lambda_s$. On each overlap, both maps
		take values in the same embedded portion of $\Lambda_s$. Interpolate
		between them in local coordinates on this surface, using fixed cutoffs.
		This gives a retraction
		on a neighbourhood of fixed width around all of $\Lambda_s$, with Lipschitz
		constant uniform in $s$.
		
		By (\ref{eqn:initial-closeness}), $\Gamma_R$ is supported in this fixed neighbourhood.
		It is easy to see that
		\[
		(\Pi_s)_\#\Gamma_R=\Lambda_s.
		\]
		On the sheets $\zeta'$ vanishes, while on the joining segments
		$\Pi_s$ has rank at most one almost everywhere. Consequently
		$\Pi_s^*\zeta'=0$ almost everywhere on $\Gamma_R$.
		
		The short geodesic homotopy from $\Pi_s$ to the identity has the required
		endpoint properties and satisfies
		\[
		|\dot H_t(x)|\le C\dist(x,\Lambda_s)
		\le C\dist(x,\mathbf C_s),\qquad |dH_t|\le C.
		\]
		Thus it has all the properties used in Step~1.
	\end{proof}

	\section{The gluing obstruction}\label{sect:gluing-obstruction}
	
	Work on $\Gamma_{R_0}$, where $R_0$ is the regular radius from
	Proposition~\ref{prop:stopping}.
	Lemma~\ref{lem:common-slice} gives extracted necks whose
	largest stopping radius is comparable to $\rho$, with
	$\rho\ge C^{-1}D^{2/\gamma}$.
	We shall prove the discrete \L{}ojasiewicz inequality in
	Proposition~\ref{prop:discrete-lojasiewicz} by combining the loop periods,
	flux balance and variation of the Reeb chord actions with the analytic
	gradient inequality.
	
	Orient the edges of $G(\Lambda)$ and write
	$a_e,\sigma_e,q_{e,\pm}(s),c_{e,\pm},P_e(s)$ for the corresponding
	local quantities. Let $G'=(V,E')$ be the subgraph of $G(\Lambda)$
	containing all vertices and exactly the edges $e$ with $\sigma_e>C(\varepsilon)\rho_*$,
	for a sufficiently large $C(\varepsilon)$, to be chosen in the proof.
	The short Reeb chord from $q_{e,-}(s)$ to $q_{e,+}(s)$ has signed action
	$\beta_e(s)$, so $q_{e,+}(s)=e^{i\beta_e(s)}q_{e,-}(s)$.
	Here $\beta_e(0)=0$, but $\beta_e(s)$ need not vanish.
	The coefficient comparison and Lemma~\ref{lem:common-slice} give
	\[
	\begin{aligned}
		|a_e|&\asymp\sigma_e^2,\quad |\beta_e(s)|\le C\sigma_e
		&&(e\in E'),\\
		|a_e|&\le C(\varepsilon)\rho_*^2 &&(e\notin E').
	\end{aligned}
	\]
	For $e\notin E'$, the stopping construction gives
	$\sigma_e\le C(\varepsilon)\rho_*$.
	The separation estimate in Claim 2 after Lemma~\ref{lem:common-slice}
	and $|q_{e,+}-q_{e,-}|=2|\sin(\beta_e(s)/2)|$ for the short Reeb chord
	therefore give
	\begin{equation}\label{eqn:small-edge-action}
		|a_e|>C\sqrt E
		\quad\Longrightarrow\quad
		|q_{e,+}-q_{e,-}|\le C(\varepsilon)\sigma_e\le C(\varepsilon)\rho_*
		\quad\Longrightarrow\quad |\beta_e(s)|\le C(\varepsilon)\rho_*.
	\end{equation}
	Proposition~\ref{prop:stopping} also gives
	$\max_{e\in E'}|a_e|\asymp\rho^2$.
	
	For an oriented simple cycle $\gamma_j$ in $G(\Lambda)$, let
	$\xi_e(\gamma_j)$ be $1$, $-1$, or $0$ according as $\gamma_j$ traverses
	$e$ with its orientation, against it, or not at all.
	
	\subsection{The logarithmic term in loop periods}
	\label{sect:logarithmic-loop-periods}
	
	On $\Gamma_{R_0}$, for each extracted edge $e$, orient the neck
	and its Reeb chord $c_s$ from the minus branch to the plus branch,
	so $\beta_e(s)=\int_{c_s}\lambda$.
	An oriented simple cycle $\gamma_j$ in $G'$ is represented by a loop
	$\gamma$ in $\Gamma_{R_0}$ crossing its Lawlor necks and joining their
	ends in the graphical region, avoiding all other puncture neighbourhoods.
	
	\begin{proposition}[Loop periods]\label{prop:direct-period}
		There is $C>0$ such that every oriented simple cycle $\gamma_j$ in $G'$ has a representative
		$\gamma$ as above satisfying
		\begin{equation}\label{eqn:loop-period}
			\int_\gamma\lambda
			=\sum_{e\in E'} \xi_e(\gamma_j)
			\{\beta_e(s)+2a_e\log|a_e|\}+e_\gamma,
		\end{equation}
		where
		\[
		\begin{split}
			|e_\gamma|
			&\le C\varepsilon\rho^2\log(1/\rho),\\
			\left|\int_\gamma\lambda\right|
			&\le C(\varepsilon)\rho_*^{-1/2}\sqrt E.
		\end{split}
		\]
	\end{proposition}
	
	\begin{proof}
		\emph{Step 1: choose the loop and bound its period.}
		For each vertex of $\gamma_j$, choose a path in the corresponding component
		of $\Sigma_{r_0}$ joining the two boundary circles associated with the
		adjacent edges. Extend each end along a curve of constant $\theta$ down
		to $r=C(\varepsilon)\sigma$, using the stopping radius of that neck.
		Map these paths to $\Gamma_{R_0}$ by the graphical parametrization and
		join each pair of endpoints through the corresponding Lawlor neck.
		Write $\gamma_{\mathrm{core}}$ for this segment between the two circles
		$r=C(\varepsilon)\sigma$, and $\gamma_{\mathrm{in}}$ and
		$\gamma_{\mathrm{out}}$ for the adjacent portions in the graphical annuli
		$C(\varepsilon)\sigma<r<r_0$ (Figure~\ref{fig:neck-period}). The resulting loop $\gamma$
		inherits its orientation from $\gamma_j$. The paths in $\Sigma_{r_0}$
		can be chosen with uniformly bounded length.
		By (\ref{eqn:liouville-interior}),
		\[
		\int_{\gamma_{\mathrm{in}}}|\lambda|
		\le C\sqrt E\int_{C(\varepsilon)\sigma}^{r_0}r^{-3/2}\,dr
		\le C\sigma^{-1/2}\sqrt E.
		\]
		The same bound holds on $\gamma_{\mathrm{out}}$. On $\gamma_{\mathrm{core}}$,
		interior estimates give the bound with constant $C(\varepsilon)$;
		this segment has length $O(C(\varepsilon)\sigma)$.
		On $\Sigma_{r_0}$ the
		bound is $C\sqrt E$. There are finitely many edges and
		$\sigma>C(\varepsilon)\rho_*$, so summing these bounds proves the second estimate.
		Moreover, $E\le\rho_*^5$ gives
		$\rho_*^{-1/2}\sqrt E\le\rho_*^2\le\rho^2$.
		Thus the total contribution of the graphical paths is $O(\rho^2)$.
		
		\smallskip
		\noindent\emph{Step 2: compute the integral along $\gamma_{\mathrm{core}}$.}
		Fix an extracted Lawlor neck corresponding to $e\in E'$, with stopping
		radius $\sigma$, and work near its ends in the coordinates
		$(p,\xi,z)$ of Subsection~\ref{sect:graphical-radial-slices}.
		Let $x_-,x_+$ be the endpoints of $\gamma_{\mathrm{core}}$ on the two
		circles $r=C(\varepsilon)\sigma$, with $z$-coordinates $u_\pm$.
		Let $\bar x_\pm\in\Lambda_s$ be their projections to the zero section.
		Join $\bar x_\pm$ to $x_\pm$ keeping $p$ fixed. Since $dp=0$,
		the $\lambda$-integral along each segment is $2u_\pm$.
		The path from $\bar x_-$ to $\bar x_+$ along the Legendrian branches
		and the Reeb chord has integral $\beta_e(s)$.
		Together with $\gamma_{\mathrm{core}}$, these paths form a closed loop
		of length $O(C(\varepsilon)\sigma)$. Coning this loop to $x_-$ gives a spanning
		surface of area $O(C(\varepsilon)\sigma^2)$. The Stokes formula and $d\lambda=2\omega$
		therefore give
		\[
		\int_{\gamma_{\mathrm{core}}}\lambda
		=2u_+-2u_-+\beta_e(s)+O(C(\varepsilon)\sigma^2).
		\]
		By the comparison following Lemma~\ref{lem:common-slice},
		(\ref{eqn:neck-coefficient}) holds with $\pm a_e$ on the corresponding
		$\pm$ ends. Since $|a_e|\asymp\sigma^2\le C\rho^2$,
		\[
		4a_e\log(C(\varepsilon)\sigma/r_0)
		=2a_e\log|a_e|+O(C(\varepsilon)\sigma^2)
		=2a_e\log|a_e|+O(C(\varepsilon)\rho^2).
		\]
		Integrating (\ref{eqn:neck-coefficient}) from $r_0$ to $C(\varepsilon)\sigma$ and
		using the bounds for $u_{R_0}$ and for $u_{R_0}-\bar u_{R_0}$
		in Proposition~\ref{prop:exterior-scalar}, we obtain
		\[
		\begin{aligned}
			\bigl|2u_+-2u_--2a_e\log|a_e|\bigr|
			&\le C\varepsilon |a_e|\log(1/|a_e|)+C(\varepsilon)\rho^2\\
			&\le C\varepsilon\rho^2\log(1/\rho).
		\end{aligned}
		\]
		In the last inequality, we take $\rho$ sufficiently small depending on
		$\varepsilon$ to absorb $C(\varepsilon)\rho^2$ into $\varepsilon\rho^2\log(1/\rho)$.
		Summing the resulting core integrals with their traversal signs and adding
		the $O(\rho^2)$ graphical contribution from Step 1 gives
		(\ref{eqn:loop-period}) and the stated error bound.
	\end{proof}
	
	\begin{figure}[t]
		\centering
		\begin{tikzpicture}[x=1.25cm,y=1cm,>=Latex,font=\small,
			directed/.style={thick,postaction={decorate},
				decoration={markings,mark=at position .55 with {\arrow{Latex}}}}]
			\foreach \s in {-1,1}{
				\begin{scope}[yscale=\s]
					\path[fill=black!9,draw=black!45,line width=.4pt]
					(-4,.75).. controls (-3.45,.4) and (-2.9,.12)..(-2.5,.12)
					.. controls (-1.25,1.47) and (1.25,1.47)..(2.5,.12)
					.. controls (2.9,.12) and (3.45,.4)..(4,.75)
					--(4,1.08)
					.. controls (3.45,.73) and (2.9,.45)..(2.5,.45)
					.. controls (1.25,1.8) and (-1.25,1.8)..(-2.5,.45)
					.. controls (-2.9,.45) and (-3.45,.73)..(-4,1.08)--cycle;
				\end{scope}
			}
			\path[even odd rule,fill=blue!9,draw=blue!45!black,line width=.45pt]
			(-2.85,1.1).. controls (-1.3,2.33) and (1.3,2.33)..(2.85,1.1)
			.. controls (3.7,.67) and (3.7,-.67)..(2.85,-1.1)
			.. controls (1.3,-2.33) and (-1.3,-2.33)..(-2.85,-1.1)
			.. controls (-3.7,-.67) and (-3.7,.67)..(-2.85,1.1)--cycle
			(-2.45,.8).. controls (-1.1,1.83) and (1.1,1.83)..(2.45,.8)
			.. controls (3.1,.53) and (3.1,-.53)..(2.45,-.8)
			.. controls (1.1,-1.83) and (-1.1,-1.83)..(-2.45,-.8)
			.. controls (-3.1,-.53) and (-3.1,.53)..(-2.45,.8)--cycle;
			\draw[directed] (-2.65,.95).. controls (-1.2,2.08) and (1.2,2.08)..(2.65,.95);
			\draw[directed] (2.65,.95).. controls (3.4,.6) and (3.4,-.6)..(2.65,-.95);
			\draw[directed] (2.65,-.95).. controls (1.2,-2.08) and (-1.2,-2.08)..(-2.65,-.95);
			\draw[directed] (-2.65,-.95).. controls (-3.4,-.6) and (-3.4,.6)..(-2.65,.95);
			\node[above] at (-.55,1.8) {$\gamma$};
			\node[above,blue!45!black] at (0,2.14) {$\Gamma_{R_0}$};
			\node at (0,1.05) {$\Lambda_s$};
			\node at (0,-1.05) {$\Lambda_s$};
			\node[above] at (1.85,1.77) {$\gamma_{\mathrm{in}}$};
			\node[below] at (1.85,-1.77) {$\gamma_{\mathrm{out}}$};
			\draw[thin] (3.33,0)--(3.72,0);
			\node[right] at (3.72,0) {$\gamma_{\mathrm{core}}$};
			\fill (2.65,.95) circle (1.2pt) (2.65,-.95) circle (1.2pt);
			\node[above right] at (2.65,.95) {$x_-$};
			\node[below right] at (2.65,-.95) {$x_+$};
			\draw[densely dotted,-{Latex[length=1.1mm,width=.8mm]},line width=.5pt]
			(2.82,.32)--(2.65,.95);
			\draw[densely dotted,-{Latex[length=1.1mm,width=.8mm]},line width=.5pt]
			(2.82,-.32)--(2.65,-.95);
			\node[right] at (3.43,.64) {$2u_-$};
			\node[right] at (3.43,-.64) {$2u_+$};
			\draw[thin,black!45] (2.73,.64)--(3.4,.64);
			\draw[thin,black!45] (2.73,-.64)--(3.4,-.64);
			\fill (2.82,.32) circle (1pt) (2.82,-.32) circle (1pt);
			\node[inner sep=1pt] at (1.93,.54) {$\bar x_-$};
			\node[inner sep=1pt] at (1.93,-.54) {$\bar x_+$};
			\draw[thin,black!45] (2.13,.5)--(2.82,.32);
			\draw[thin,black!45] (2.13,-.5)--(2.82,-.32);
			\draw[black!65,thin]
			(2.82,.32).. controls (2.70,.14) and (2.58,.12)..(2.5,.12);
			\draw[black!65,thin]
			(2.5,-.12).. controls (2.58,-.12) and (2.70,-.14)..(2.82,-.32);
			\draw[densely dashed,-{Latex[length=1.1mm,width=.8mm]},line width=.4pt]
			(2.5,.12)--(2.5,-.12);
			\draw[densely dashed,-{Latex[length=1.1mm,width=.8mm]},line width=.4pt]
			(-2.5,-.12)--(-2.5,.12);
			\fill (2.5,.12) circle (.9pt) (2.5,-.12) circle (.9pt);
			\fill (-2.5,.12) circle (.9pt) (-2.5,-.12) circle (.9pt);
			\node[left,inner sep=2pt] at (2.5,0) {$c_s$};
			\node at (0,0) {$\displaystyle\int_{c_s}\lambda=\beta_e(s)$};
		\end{tikzpicture}
		\caption{The loop $\gamma$ on $\Gamma_{R_0}$ (blue) and the comparison
			surface $\Lambda_s$ (grey). The endpoints $x_\pm$ lie at
			$r=C(\varepsilon)\sigma$. The dotted paths have $\lambda$-integrals
			$2u_\pm$; the short dashed Reeb chords are not part of $\gamma$.}
		\label{fig:neck-period}
	\end{figure}
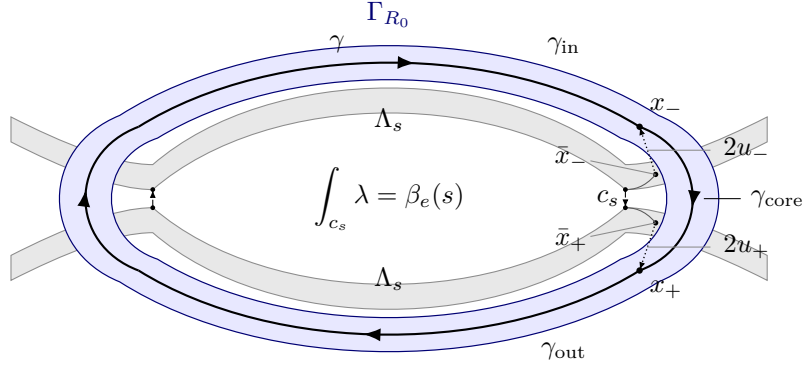
	
	\subsection{Reeb chord actions and flux balance}
	\label{sect:variation-chord-actions}
	
	For each edge $e$, let $p_{e,\pm}\in\Sigma$ be the corresponding
	preimages of the points
	$q_{e,\pm}(s)=F_s(p_{e,\pm})\in\Lambda_s$.
	
	\begin{proposition}[Action variation and flux balance]
		\label{prop:chord-action-variation}
		A tangent vector $\dot s\in T_s\mathcal M$ represented by the Jacobi
		normal field $2hJF_s+J\nabla h$ satisfies
		\begin{equation}\label{eqn:action-derivative}
			d\beta_e(s)[\dot s]=2\{h(p_{e,+})-h(p_{e,-})\}.
		\end{equation}
		The coefficients $a_e$ satisfy
		\begin{equation}\label{eqn:chord-action-balance}
			\left\|\sum_{e\in E}a_e\,d\beta_e(s)\right\|_{T_s^*\mathcal M}
			\le C(\rho^4\log(1/\rho)+\sqrt E).
		\end{equation}
		At every vertex $v$,
		\begin{equation}\label{eqn:matrix-flux}
			\left\|\sum_{e\text{ enters }v}a_e(P_e(s)-I/3)
			-\sum_{e\text{ leaves }v}a_e(P_e(s)-I/3)\right\|
			\le C(\rho^4\log(1/\rho)+\sqrt E).
		\end{equation}
	\end{proposition}
	
	\begin{proof}
		\emph{Step 1: variation of the action.}
		Since the restriction of $d\lambda$ to $S^5$ annihilates the Reeb
		direction, variation of the chord integral leaves only the endpoint terms. Their tangential parts vanish
		on the Legendrian branches, giving (\ref{eqn:action-derivative}).
		
		\smallskip
		\noindent\emph{Step 2: flux balancing.}
		Take a Jacobi function $h$ with $\|h\|_{C^3(\Sigma)}\le1$, and let
		$\sigma_l$ be the cutoff radius chosen at the puncture $p_l\in\Sigma$.
		Rearranging the slice equation from Step 1 of
		Proposition~\ref{prop:exterior-scalar} and absorbing the nonlinear terms
		into $\mathcal N$ gives, on
		$\Sigma\setminus\bigcup_l B_{C\sigma_l}^{\Sigma}(p_l)$
		with $C$ sufficiently large and independent of $\varepsilon$,
		\[
		\Delta_\Sigma u_{R_0}+6u_{R_0}=\mathcal N,\qquad
		|\mathcal N|\le C\left(\frac{\sigma^4}{r^4}
		+\frac{\rho^4}{r^2}+r^{-5/2}\sqrt E\right).
		\]
		The graphical estimate
		$|u_{R_0}|+|du_{R_0}|+|D^2u_{R_0}|\le C(\sigma^2/r^2+\rho^2/r)$
		ensures smallness for $r\ge C\sigma$, since $\sigma\ge\rho^p$ and $p<2$.
		Squaring the right-hand side bounds the quadratic Taylor remainder and gives the first
		two terms in the bound for $\mathcal N$.
		The radial and Liouville-form terms give the last term,
		as in Step 1 of Proposition~\ref{prop:exterior-scalar}.
		In particular, $|\mathcal N|\le C(\rho^4+\sqrt E)$ on $\Sigma_{r_0}$.
		
		Let $H_l$ be trace-free Hermitian matrices with $\|H_l\|\le C$ such that
		$h_l=\langle H_lF_s,F_s\rangle$ agrees with $h$ to first order at $p_l$.
		Then
		$|h-h_l|+r|d(h-h_l)|\le Cr^2$.
		Apply Green's formula to the Jacobi function $h-h_l$ on
		$B_{r_0}^{\Sigma}(p_l)\setminus B_{C\sigma_l}^{\Sigma}(p_l)$.
		The graphical bounds give
		\begin{equation}\label{eqn:green-comparison-bound}
			\begin{aligned}
				&\left|\int_{r=r_0}
				\bigl((h-h_l)\partial_r u_{R_0}-u_{R_0}\partial_r(h-h_l)\bigr)\,ds\right|\\
				&\quad\le\left|\int_{B_{r_0}^{\Sigma}(p_l)\setminus B_{C\sigma_l}^{\Sigma}(p_l)}
				(h-h_l)\mathcal N\,d\mathrm{Area}\right|\\
				&\qquad+\left|\int_{r=C\sigma_l}
				\bigl((h-h_l)\partial_r u_{R_0}-u_{R_0}\partial_r(h-h_l)\bigr)\,ds\right|\\
				&\quad\le C\int_{C\sigma_l}^{r_0}
				\left(\frac{\sigma_l^4}{r}+\rho^4r+r^{1/2}\sqrt E\right)\,dr
				+C\bigl(\sigma_l^4+\rho^2\sigma_l^2\log(1/\sigma_l)\bigr)\\
				&\quad\le C\bigl(\rho^4\log(1/\rho)+\sqrt E\bigr),
			\end{aligned}
		\end{equation}
		where the last inequality uses $\rho_*\le\sigma_l\le C\rho$.
		Now apply Green's formula to $h$ on $\Sigma_{r_0}$:
		\[
		\sum_l\int_{r=r_0}(h\partial_r u_{R_0}-u_{R_0}\partial_r h)\,ds
		=-\int_{\Sigma_{r_0}}h\mathcal N\,d\mathrm{Area}
		=O(\rho^4+\sqrt E).
		\]
		By (\ref{eqn:green-comparison-bound}), replacing $h$ by $h_l$ in each boundary
		integral changes the sum by at most $C(\rho^4\log(1/\rho)+\sqrt E)$.
		Pull back the flux forms by the graphical parametrization and orient
		each circle as the boundary of $B_{r_0}^{\Sigma}(p_l)$.
		Then the $O(\rho^2)$ graphical bound at $r=r_0$, together with
		(\ref{eqn:contact-equation}) and (\ref{eqn:liouville-interior}), gives
		\[
		\int_{r=r_0}(h_l\partial_r u_{R_0}-u_{R_0}\partial_r h_l)\,ds
		=\frac13\int_{r=r_0}\eta_{H_l}+O(\rho^4+\sqrt E).
		\]
		Consequently,
		\begin{equation}\label{eqn:green-flux-bound}
			\left|\sum_l\int_{\partial B_{r_0}^{\Sigma}(p_l)}\eta_{H_l}\right|
			\le C\bigl(\rho^4\log(1/\rho)+\sqrt E\bigr).
		\end{equation}
		
		\smallskip
		\noindent\emph{Step 3: pair the endpoint fluxes.}
		We evaluate the flux integrals in (\ref{eqn:green-flux-bound}) in pairs,
		choosing the two matrices from Step 2 at each edge simultaneously.
		Fix an edge $e$, with $p_\pm\in\Sigma$ and $q_\pm=F_s(p_\pm)$, and recall
		$P=q_-q_-^*=q_+q_+^*$.
		Write $h_\pm=\langle H_\pm F_s,F_s\rangle$ for the two functions $h_l$.
		The requirements on $h_+$ and $h_-$ are
		\[
		h_\pm(p_\pm)=h(p_\pm),\qquad
		dh_-(p_\pm)=dh_+(p_\pm)=dh(p_\pm).
		\]
		Here the differentials are taken on $\Lambda_s$ and do not include
		the radial direction of the Lagrangian cone.
		The transversality of the two tangent planes in $(\mathbb C q_-)^\perp$
		allows us to choose $H_-$ so that $dh_-(p_\pm)=dh(p_\pm)$,
		while independently prescribing $h_-(p_-)=h(p_-)$.
		Since $q_\pm$ lie on the same Hopf fibre,
		$h_-(p_+)=h(p_-)$.
		The function $\frac32\langle(P-I/3)F_s,F_s\rangle$ has value $1$
		and zero differential at $p_\pm$. Hence the choice
		\[
		H_+=H_-+\frac32\{h(p_+)-h(p_-)\}(P-I/3)
		\]
		satisfies the stated requirements.
		The uniform transversality of the two tangent planes gives
		$\|H_-\|+\|H_+\|\le C$, with $C$ independent of $s$.
		
		By Claim 1 in the flux interpretation following
		Lemma~\ref{lem:common-slice}, moving the circles from $r=C(\varepsilon)\sigma_e$
		to $r=r_0$ changes each flux by $O(\sqrt E)$.
		Thus (\ref{eqn:paired-boundary-flux}) and (\ref{eqn:flux-coefficient})
		give the contribution of this edge to
		(\ref{eqn:green-flux-bound}) as
		\[
		6\pi a_e\{h(p_+)-h(p_-)\}+O(\sqrt E).
		\]
		Restoring the edge indices, summing and rescaling $h$ gives
		\begin{equation}\label{eqn:flux-balance}
			\left|\sum_{e\in E}a_e\{h(p_{e,+})-h(p_{e,-})\}\right|
			\le C\|h\|_{C^3(\Sigma)}\bigl(\rho^4\log(1/\rho)+\sqrt E\bigr).
		\end{equation}
		
		Fix a smooth norm on $T\mathcal M$. For $\|\dot s\|=1$, the corresponding
		Jacobi functions have uniformly bounded $C^3(\Sigma)$ norm by elliptic
		estimates. Combining
		(\ref{eqn:action-derivative}) with (\ref{eqn:flux-balance}) gives
		(\ref{eqn:chord-action-balance}).
		
		\smallskip
		\noindent\emph{Step 4: balance at each vertex.}
		Fix a vertex $v$ and a trace-free Hermitian matrix $H$.
		Consider the Jacobi function defined by $h=\langle HF_s,F_s\rangle$ on
		$\Sigma_v$ and $h=0$ on the other components. It satisfies
		$\|h\|_{C^3(\Sigma)}\le C\|H\|$.
		The points $q_{e,-}(s)$ and $q_{e,+}(s)$ have the same projection matrix $P_e(s)$, so
		\[
		\sum_{e\in E}a_e\{h(p_{e,+})-h(p_{e,-})\}
		=\sum_{e\text{ enters }v}a_e\operatorname{tr}(HP_e(s))
		-\sum_{e\text{ leaves }v}a_e\operatorname{tr}(HP_e(s)).
		\]
		An edge from $v$ to itself contributes to both sums with opposite signs.
		The flux balance (\ref{eqn:flux-balance}) bounds this
		expression by $C\|H\|(\rho^4\log(1/\rho)+\sqrt E)$ in absolute value.
		Since $\operatorname{tr}H=0$, we may replace $P_e(s)$ by $P_e(s)-I/3$.
		Taking the supremum over $\|H\|\le1$ controls the norm of the trace-free
		matrix in (\ref{eqn:matrix-flux}), proving the bound.
	\end{proof}
	
	\subsection{Application of the \L{}ojasiewicz gradient inequality}
	\label{sect:noncancellation}
	
	\begin{lemma}[Weighted \L{}ojasiewicz inequality]
		\label{lem:weighted-action-lojasiewicz}
		Let $\gamma_j$ be distinct oriented simple cycles of $G(\Lambda)$, and put
		$g_j(s)=\sum_{e\in \gamma_j}\xi_e(\gamma_j)\beta_e(s)$.
		There are $C>0$ and $0<\theta<1$ such that,
		for sufficiently small $s$ and real coefficients $|w_j|\le1$,
		\[
		\left|\sum_jw_jg_j(s)\right|^\theta
		\le C\left(
		\left\|\sum_jw_j\,dg_j(s)\right\|
		+\left(\sum_j|w_j|\,|g_j(s)|^2\right)^{1/2}\right).
		\]
		The constants are uniform over the choices of cycles and orientations.
	\end{lemma}
	
	\begin{proof}
		By Jacobi integrability, we may take $\Lambda_s$ to depend real
		analytically on $s$.
		Since the two local sheets intersect transversely after Hopf projection,
		the short Reeb chords and their actions depend analytically on $s$. Thus each $g_j$ is analytic
		and $g_j(0)=0$.
		
		For fixed signs $\epsilon_j\in\{-1,1\}$, set
		$f(s,z)=\sum_j\epsilon_j z_j^2g_j(s)$, with $s,z$ independent.
		Since $f(0,z)=0$, the \L{}ojasiewicz gradient inequality and compactness give
		\[
		|f(s,z)|^\theta\le C\bigl(\|d_sf(s,z)\|+|d_zf(s,z)|\bigr)
		\]
		uniformly for small $s$ and $|z_j|\le1$.
		Writing $w_j=\epsilon_j z_j^2$ gives the stated inequality.
		There are only finitely many choices of cycles, orientations and signs,
		so $C$ and $\theta$ can be chosen uniformly.
	\end{proof}
	
	\begin{proposition}[Gluing obstruction]\label{prop:assembly}
		Assume that $\Lambda$ satisfies the intersection condition in
		Definition~\ref{def:graph-condition}, and let $\theta$ be as in
		Lemma~\ref{lem:weighted-action-lojasiewicz}. Suppose
		\[
		\max\{20/3,4+4\theta\}<q_0<8.
		\]
		Then there are $\delta,\rho_0>0$ such that, if
		(\ref{eqn:initial-closeness}) holds and the stopping radius
		$\rho$ supplied by Proposition~\ref{prop:stopping} satisfies
		$0<\rho<\rho_0$, then
		\begin{equation}\label{eqn:gluing-obstruction}
			E\ge\rho^{q_0}.
		\end{equation}
	\end{proposition}
	
	\begin{proof}
		We will choose $1<p<q_0/5$ at the end of the proof.
		If $E\ge\rho^{q_0}$, there is nothing to prove. We therefore assume
		$E<\rho^{q_0}<\rho^{5p}=\rho_*^5$, so the preceding neck construction
		applies. Write $\Delta=\rho^4\log(1/\rho)+\sqrt E$;
		then $\Delta=o(\rho^{2p})$ as $\rho\to0$.
		
		\smallskip
		\noindent\emph{Step 1: reduce the vertex balance to cycle coefficients.}
		Choose a spanning forest of $G(\Lambda)$. Adding each omitted edge produces a simple
		cycle; no two of these cycles share an edge by the intersection condition
		in Definition~\ref{def:graph-condition}.
		Denote them by $\gamma_j$ and orient their edges cyclically, so
		$\xi_e(\gamma_j)=1$ for $e\in \gamma_j$.
		
		\emph{Claim.} Let $e_j$ be the edge of $\gamma_j$ outside the spanning forest.
		Then
		\[
		|a_e-a_{e_j}|\le C\Delta\quad(e\in \gamma_j),\qquad
		|a_e|\le C\Delta\quad(e\text{ a bridge}).
		\]
		To prove the first estimate, remove an edge $e\in \gamma_j\setminus\{e_j\}$
		from the forest. This splits its tree into two components; let $S$
		be the vertex set of the component into which $e$ points.
		Only $e$ and $e_j$ join $S$ to its complement in $G(\Lambda)$, because
		any further edge, together with the unique path in the original forest
		between its endpoints, would produce a second simple cycle containing
		$e$, contrary to the intersection condition in
		Definition~\ref{def:graph-condition}.
		
		Sum the vertex balances (\ref{eqn:matrix-flux}) over $S$.
		Every edge with both endpoints in $S$ contributes its matrix once with
		each sign, so all such terms cancel exactly. Since $\gamma_j$ is oriented
		cyclically, $e$ enters $S$ and $e_j$ leaves it. Consequently,
		\[
		\begin{aligned}
			&a_e(P_e-I/3)-a_{e_j}(P_{e_j}-I/3)\\
			&\quad=\sum_{v\in S}\left(
			\sum_{f\text{ enters }v}a_f(P_f-I/3)
			-\sum_{f\text{ leaves }v}a_f(P_f-I/3)\right)
			=O(\Delta).
		\end{aligned}
		\]
		The number of vertices is fixed, so the sum of the errors is still
		$O(\Delta)$.
		Since the middle eigenvalue of $a(P-I/3)$ is $-a/3$ for either sign
		of $a$, eigenvalue continuity gives the first estimate.
		Removing a bridge from $G(\Lambda)$ separates its endpoints.
		Summing over the component containing one endpoint leaves only that
		edge in the boundary sum, giving the estimate.
		
		Set $t_j=a_{e_j}$, so $a_e=t_j+O(\Delta)$ on $\gamma_j$ and
		$|t_j|\le C\rho^2$.
		Discard the bridges and cycles with $|t_j|\le C\Delta$, for a sufficiently
		large fixed $C$. Their contribution to (\ref{eqn:chord-action-balance}) is
		$O(\Delta)$. All sums below are over the remaining cycles, with $g_j$
		as in Lemma~\ref{lem:weighted-action-lojasiewicz}.
		Replacing $a_e$ by $t_j$ in (\ref{eqn:chord-action-balance}) gives
		\begin{equation}\label{eqn:cycle-derivative}
			\left\|\sum_jt_j\,dg_j(s)\right\|\le C\Delta.
		\end{equation}
		Each $g_j$ includes all edges of its cycle.
		
		\smallskip
		\noindent\emph{Step 2: estimate the weighted sums of chord actions.}
		On each remaining cycle, $|a_e|\asymp|t_j|$ for every $e\in \gamma_j$.
		We distinguish two cases.
		
		If $\gamma_j\subset E'$, the loop period formula in
		Proposition~\ref{prop:direct-period} gives
		\begin{equation}\label{eqn:cycle-period}
			g_j(s)=2\#\gamma_j\,t_j\log(1/|t_j|)
			+O\!\left(\varepsilon\rho^2\log(1/\rho)\right).
		\end{equation}
		Here replacing $a_e$ by $t_j$ costs $O(\Delta\log(1/\rho))$,
		and the loop period is bounded by $C(\varepsilon)\rho_*^{-1/2}\sqrt E$.
		Since $E<\rho^{5p}$ and $p>1$, both are $o(\rho^2\log(1/\rho))$
		and are absorbed by taking $\rho$ sufficiently small in terms of
		$\varepsilon$.
		
		If $\gamma_j\not\subset E'$, an edge outside $E'$ has coefficient
		$O(C(\varepsilon)\rho_*^2)$, so $|t_j|$ and all $|a_e|$ on this cycle are
		$O(C(\varepsilon)\rho_*^2)$. Hence $|\beta_e(s)|\le C(\varepsilon)\rho_*$ for every edge:
		for $e\in E'$ this follows from $|\beta_e(s)|\le C\sqrt{|a_e|}$;
		for $e\notin E'$, it follows from (\ref{eqn:small-edge-action}),
		since $|a_e|>C\sqrt E$ on every remaining cycle.
		Thus
		\begin{equation}\label{eqn:mixed-cycle-bound}
			|t_j|\le C(\varepsilon)\rho^{2p},\qquad |g_j(s)|\le C(\varepsilon)\rho^p
			\qquad(\gamma_j\not\subset E').
		\end{equation}
		Finally, since $\max_{e\in E'}|a_e|\asymp\rho^2$ and $\Delta=o(\rho^2)$,
		at least one remaining cycle has $|t_j|\asymp\rho^2$.
		By (\ref{eqn:mixed-cycle-bound}) and $p>1$, this cycle is contained in $E'$.
		
		The cycles not contained in $E'$ contribute $O(C(\varepsilon)\rho^{3p-2})$ to
		$\rho^{-2}\sum_jt_jg_j(s)$ by
		(\ref{eqn:mixed-cycle-bound}). For the other cycles,
		(\ref{eqn:cycle-period}) gives
		\[
		\rho^{-2}\sum_jt_jg_j(s)
		=2\rho^{-2}\sum_{\gamma_j\subset E'}\#\gamma_j\,t_j^2\log(1/|t_j|)
		+O\!\left(\varepsilon\rho^2\log(1/\rho)+C(\varepsilon)\rho^{3p-2}\right).
		\]
		For small $\rho$, each summand $\#\gamma_j\,t_j^2\log(1/|t_j|)$ on the right
		is positive, since $0<|t_j|\le C\rho^2<1$.
		The cycle with $|t_j|\asymp\rho^2$
		contributes at least $C^{-1}\rho^2\log(1/\rho)$ in absolute value.
		If $p>4/3$, then $C(\varepsilon)\rho^{3p-2}=o(\rho^2)$
		for each fixed $\varepsilon$. Choosing $\varepsilon$ and then $\rho$
		small therefore gives
		\[
		\left|\rho^{-2}\sum_jt_jg_j(s)\right|
		\ge C^{-1}\rho^2\log(1/\rho).
		\]
		
		The same estimates also give
		\[
		\rho^{-1}\left(\sum_j|t_j|\,|g_j(s)|^2\right)^{1/2}
		\le C\rho^2\log(1/\rho)+C(\varepsilon)\rho^{2p-1}.
		\]
		
		\smallskip
		\noindent\emph{Step 3: apply the weighted \L{}ojasiewicz inequality.}
		Apply Lemma~\ref{lem:weighted-action-lojasiewicz} with
		$w_j=t_j/\max_k|t_k|$. Since $\max_k|t_k|\asymp\rho^2$, it gives
		\[
		\begin{aligned}
			(\rho^2\log(1/\rho))^\theta
			&\le C\left(\rho^{-2}\Delta+\rho^2\log(1/\rho)\right)
			+C(\varepsilon)\rho^{2p-1}\\
			&\le C\left(\rho^2\log(1/\rho)+\rho^{q_0/2-2}\right)
			+C(\varepsilon)\rho^{2p-1},
		\end{aligned}
		\]
		by (\ref{eqn:cycle-derivative}), Step 2 and $E<\rho^{q_0}$.
		We now choose
		\[
		\max\{4/3,\theta+1/2\}<p<q_0/5.
		\]
		The assumed range of $q_0$ and $\theta<1$ make this interval nonempty.
		With $p$ fixed, choose $\varepsilon$ sufficiently small as in Step 2.
		Then $C(\varepsilon)$ is fixed, and $2p-1>2\theta$ and
		$q_0/2-2>2\theta$ make the right-hand side
		$o((\rho^2\log(1/\rho))^\theta)$, a contradiction for small $\rho$.
		All smallness requirements are met by choosing $\delta,\rho_0$ sufficiently small.
		Hence $E\ge\rho^{q_0}$, proving
		(\ref{eqn:gluing-obstruction}).
	\end{proof}
	
	\begin{example}[The two-plane case]\label{ex:antipodal-cancellation}
		The two Reeb chords are antipodal images. Since $x\mapsto-x$ preserves
		$\lambda$, their actions agree
		when oriented from the first sphere to the second. The cycle traverses
		them in opposite directions, so $g(s)=0$ for every $s$. The flux and
		period estimates then give the obstruction directly.
	\end{example}

	\subsection{The discrete \L{}ojasiewicz estimate}
	
	\begin{proposition}[Discrete \L{}ojasiewicz estimate]
		\label{prop:discrete-lojasiewicz}
		Assume that $\Lambda$ satisfies the intersection condition in
		Definition~\ref{def:graph-condition}.
		There are $q\in(1,2)$ and $\delta,C>0$ such that, whenever
		$L_{0,R}=(\lambda_{0,R})_\#L$ satisfies (\ref{eqn:initial-closeness}),
		\[
		D(R)^{2q}\le CE(R).
		\]
		If $\mathbf C_0$ is a tangent cone of $L$ at the origin, then also
		\begin{equation}\label{eqn:discrete-lojasiewicz}
			m(R)^q\le C\{m(4R)-m(R/4)\}.
		\end{equation}
	\end{proposition}
	
	\begin{proof}
		By rescaling, it suffices to take $R=1$.
		Fix $q_0$ as in Proposition~\ref{prop:assembly}, and choose
		\[
		q_0/2<\gamma<4,\qquad q=q_0/\gamma\in(1,2).
		\]
		We may assume $D>0$ and $E\le\varepsilon_1D^2$, with $\varepsilon_1$
		as in Proposition~\ref{prop:stopping}; otherwise the estimate is immediate
		for $D\le1$.
		Proposition~\ref{prop:stopping} gives $\rho\ge C^{-1}D^{2/\gamma}$.
		Allard regularity at fixed scales ensures that $\rho$ is below the
		threshold in Proposition~\ref{prop:assembly} once $\delta$ is sufficiently small.
		That proposition gives
		\[
		E\ge\rho^{q_0}\ge C^{-1}D^{2q}.
		\]
		
		When $\mathbf C_0$ is a tangent cone, Proposition~\ref{prop:density}
		and the uniform bound for $E$ give
		\[
		m(1)^q\le C(D^{2q}+E^q)\le CE
		=C\{m(4)-m(1/4)\}.
		\]
		Rescaling proves (\ref{eqn:discrete-lojasiewicz}).
	\end{proof}
	
	\begin{remark*}
		The strict inequality $q<2$ is essential for the radial summability
		in Lemma~\ref{lem:numerical}. Since $q=q_0/\gamma$ and $\gamma<4$,
		we need $q_0<2\gamma<8$. If only $E\ge\rho^8$ were available,
		the same argument would give $m(1)^{8/\gamma}\le CE$, with exponent
		greater than $2$. The bound $O(\rho^4\log(1/\rho))$ on the flux error
		suffices for the argument in Step~3 of Proposition~\ref{prop:assembly}
		with any \L{}ojasiewicz exponent $\theta<1$. A weaker bound $O(\rho^3)$ need not suffice when $\theta>1/2$.
	\end{remark*}
	
	\begin{proof}[Proof of Theorem~\ref{thm:rigidity}]
		Let $\Lambda'$ be a special Legendrian integral two-cycle sufficiently
		close to $\Lambda$ in the varifold topology. Varifold convergence and
		monotonicity give (\ref{eqn:initial-closeness}) for $\mathbf C(\Lambda')$.
		Since $E=0$, Proposition~\ref{prop:discrete-lojasiewicz} gives $D=0$,
		so the cone is supported on $\mathbf C_s$.
		The initial graphical description implies multiplicity one on a regular
		patch of each component, and constancy gives
		$\mathbf C(\Lambda')=\mathbf C_s$ as currents. Taking radial slices
		gives $\Lambda'=\Lambda_s$.
	\end{proof}
	
	\section{Uniqueness of the tangent cone}\label{sect:uniqueness-tangent-cone}
	
	We use Proposition~\ref{prop:discrete-lojasiewicz} to prove summability
	of the radial variation, then show that all sufficiently small rescalings
	satisfy (\ref{eqn:initial-closeness}). Compare \cite[Section~1]{CM} and
	\cite[Theorem~6.7]{Sze20}.
	
	\subsection{Summability of radial variation}
	
	Let $R_j=4^{-j}R_0$ and $\alpha=1-q/2>0$. All statements initially
	concern a finite range of scales on which (\ref{eqn:discrete-lojasiewicz})
	is known; no continuation below the last scale is assumed.
	
	\begin{lemma}\label{lem:numerical}
		If $m(4R_0)\le1$ and (\ref{eqn:discrete-lojasiewicz}) holds at $R_j$ for
		$0\le j\le n$, then
		\begin{equation}\label{eqn:radial-summability}
			\sum_{j=0}^n\sqrt{m(4R_j)-m(R_j/4)}
			\le C m(4R_0)^\alpha.
		\end{equation}
		The estimate also holds when the sum starts at any later scale $R_j$.
	\end{lemma}
	
	\begin{proof}
		Write $u=m(4R)$ and $v=m(R/4)$. If $v\le u/2$, then
		$\sqrt{u-v}\le\sqrt u\le u^\alpha\le C(u^\alpha-v^\alpha)$.
		Otherwise $m(R)\ge v>u/2$, so (\ref{eqn:discrete-lojasiewicz}) gives
		$u-v\ge C^{-1}u^q$. Concavity then gives
		$u^\alpha-v^\alpha\ge\alpha u^{-q/2}(u-v)
		\ge C^{-1}\sqrt{u-v}$. Taking a telescoping sum gives the result.
	\end{proof}
	
	\subsection{Proof of tangent-cone uniqueness}
	
	Choose $\delta>0$ small enough that
	Proposition~\ref{prop:discrete-lojasiewicz} applies and
	(\ref{eqn:initial-closeness}) ensures that a best approximating cone
	can be chosen with $s\in\operatorname{int}\mathcal M$.
	This follows from the minimizing property and the normal graphs on fixed
	regular regions, as at the beginning of Section~\ref{sect:graphical-extension-lawlor-neck-extraction}.
	
	For the argument below, measure height from the fixed cone by
	\[
	H(R)^2=R^{-5}\int_{L\cap(B_{8R}\setminus\overline B_{R/8})}
	\dist(x,\mathbf C_0)^2\,d\|L\|(x).
	\]
	By a stationarity calculation similar to that in Step 1 of
	Proposition~\ref{prop:stopping}, we obtain the Poincar\'e estimate
	\[
	|H(t)-H(R)|\le C\sqrt{m(8R)-m(t/8)},
	\qquad R/4\le t\le R.
	\]
	
	Choose $\tau>0$ so that $C\tau^4<\delta^2/2$. On the regular region
	where $\dist(x,\mathcal R_0)>\tau|x|$, choose
	$0<\varepsilon<1$ sufficiently small that the graphicality
	criterion at the beginning of
	Section~\ref{sect:graphical-extension-lawlor-neck-extraction} gives graphs of
	$C^1$ norm less than $\tau$, and require
	\[
	C\varepsilon^{2\alpha}<\delta^2/2,\qquad
	64\varepsilon<\delta.
	\]
	Finally, choose $R_0$ along the prescribed tangent sequence so that
	\[
	H(R_0)<\varepsilon,\qquad m(16R_0)<\varepsilon,
	\]
	and $L_{0,R_0}$ satisfies (\ref{eqn:initial-closeness}) and is given by
	multiplicity-one normal graphs of $C^1$ norm less than $\varepsilon$
	over the regular sheets with $1/8<|x|<8$ and
	$\dist(x,\mathcal R_0)>(\tau/2)|x|$.
	
	\emph{Inductive claim.} The conditions
	(\ref{eqn:initial-closeness}) hold for every $R\in[R_n,R_0]$.
	
	The starting conditions give the case $n=0$. Suppose the claim holds
	down to $R_n$. For any $R\in[R_{n+1},R_n]$, the Poincar\'e estimate and
	Lemma~\ref{lem:numerical} give
	\[
	H(R)\le H(R_0)+C\sum_{k=0}^n\sqrt{E(R_k)}+C\sqrt\varepsilon
	\le C\varepsilon^\alpha.
	\]
	The finitely many endpoint terms arising from the larger annuli in
	$H$ are bounded by $C\sqrt\varepsilon$, since $m(16R_0)<\varepsilon$.
	The same bound holds at every intermediate scale between $R$ and $R_0$.
	
	Using monotonicity and the inductive hypothesis, we apply Allard regularity
	to obtain
	normal graphs $V$ over the regular region with $1/4<|x|<4$ and
	$\dist(x,\mathcal R_0)>\tau|x|$, satisfying
	\[
	|V|+|\nabla V|\le C_\tau H(R)\le C_\tau\varepsilon^\alpha<\tau.
	\]
	These graphs bound the reverse height on this region by $CH(R)^2$.
	The omitted part of $\mathbf C_0$ has mass
	$O(\tau^2)$ in $B_4\setminus\overline B_{1/4}$, and each of its points
	lies within $C\tau$ of the adjacent graphs.
	Consequently,
	\[
	\begin{aligned}
		&\int_{L_{0,R}\cap(B_4\setminus\overline B_{1/4})}
		\dist(x,\mathbf C_0)^2\,d\|L_{0,R}\|(x)\\
		&\qquad+\int_{\mathbf C_0\cap(B_4\setminus\overline B_{1/4})}
		\dist(x,\operatorname{spt}L_{0,R})^2\,d\|\mathbf C_0\|(x)
		\le CH(R)^2+C\tau^4
		\le C\varepsilon^{2\alpha}+C\tau^4.
	\end{aligned}
	\]
	By the choices of $\tau$ and $\varepsilon$, this is less than $\delta^2$.
	The mass condition follows from
	\[
	\|L_{0,R}\|(B_4)-\|\mathbf C_0\|(B_4)
	=64m(4R)\le64\varepsilon<\delta.
	\]
	This proves the inductive step.
	
	Thus (\ref{eqn:initial-closeness}) holds at every sufficiently small
	scale. Proposition~\ref{prop:discrete-lojasiewicz} places every tangent
	cone in the family $\mathbf C_s$, as in the proof of
	Theorem~\ref{thm:rigidity}. Since $\delta$ may be chosen arbitrarily
	small, every tangent cone equals $\mathbf C_0$.

\end{document}